\documentclass{article}
\usepackage{amsmath}
\usepackage{amsfonts}
\usepackage{amssymb}
\usepackage{graphicx}
\usepackage{amsthm}
\usepackage{mathrsfs}
\usepackage{pgf,tikz}
\usepackage{multirow}
\def \Cay {\mathrm{Cay}}
\def \cA {{\cal A}}
\def \cB {{\cal B}}
\def \cC {{\cal C}}

\def \cF {{\cal F}}

\def \Z {\mathbb Z}

\newcommand{\ol}{\overline}

\def\Z{\mathbb{Z}}

\def\Kmn{K_{m}[n]}

\newcommand{\comment}[1]{}

\def \cC {{\cal C}}
\def \sC {{\mathscr C}}

\newtheorem{defini}{Definition}[section]
\newtheorem{prop}[defini]{Proposition}
\newtheorem{lemma}[defini]{Lemma}

\newtheorem{cor}[defini]{Corollary}
\newtheorem{thm}[defini]{Theorem}

\theoremstyle{definition}
\newtheorem{rem}[defini]{Remark}
\newtheorem{ex}[defini]{Example}
\newtheorem{nota}[defini]{Notation}
\begin{document}
\title{Cyclic cycle systems of the complete multipartite graph}
\author{Andrea Burgess \thanks{andrea.burgess@unb.ca}\\
Department of Mathematics and Statistics,\\
University of New Brunswick \\Saint John, Canada \and
Francesca Merola  \thanks{merola@mat.uniroma3.it} \\
Dipartimento di Matematica e Fisica,\\ Universit\`a Roma Tre,\\Rome, Italy \and
Tommaso Traetta  \thanks{tommaso.traetta@unibs.it }\\
DICATAM, \\ Universit\`a di Brescia,\\ Brescia, Italy}

\maketitle

\begin{abstract}
In this paper, we study the existence problem for cyclic $\ell$-cycle decompositions of the graph $\Kmn$, the complete multipartite graph with $m$ parts of size $n$, and give necessary and sufficient conditions for their existence in the case that $2\ell \mid (m-1)n$.
\end{abstract}

\section{Introduction}

In this paper, we consider the problem of decomposing the complete multipartite graph into cycles.  We use the notation $\Kmn$ to denote the complete multipartite graph with $m$ parts of size $n$.  Note that if $n=1$, then $K_m[1]$ is isomorphic to the complete graph $K_m$ on $m$ vertices, while $K_m[2]$ is isomorphic to $K_{2m}-I$, the complete graph on $2m$ vertices with the edges of a 1-factor $I$ removed.  We denote by $\mathscr{C}_{\ell}$ a cycle of length $\ell$ (briefly, an $\ell$-cycle), and by $(c_0, c_1, \ldots, c_{\ell-1})$ the $\ell$-cycle whose edges are $\{c_0, c_1\}$, $\{c_1, c_2\}$, $\ldots$, $\{c_{\ell-1},c_0\}$.  

We say that a graph $\Gamma$ is {\em decomposed} into subgraphs $\Gamma_1, \Gamma_2, \ldots, \Gamma_t$, if the edge sets of the $\Gamma_i$ partition the edges of $\Gamma$. If $\Gamma_1 \cong \Gamma_2 \cong \cdots \cong \Gamma_t \cong H$, then we speak of an $H$-decomposition of $\Gamma$.  A $\mathscr{C}_{\ell}$-decomposition of a graph $\Gamma$ is also referred to as an {\em $\ell$-cycle system of $\Gamma$}. The problem of decomposing $K_m$ if $m$ is odd, or $K_m-I$ if $m$ is even, into cycles of fixed length $\ell$ has a long history (see~\cite[Chapter 8]{BlueBook} and \cite[Chapter VI.12]{handbook}) until its solution in~\cite{AlspachGavlas, Sajna} (see also \cite{bur03}).  
\begin{thm}[\cite{AlspachGavlas, Sajna}]
There is a $\mathscr{C}_{\ell}$-decomposition of $K_m$,  $m$ odd, if and only if $3 \leq \ell \leq m$ and $\ell \mid \binom{m}{2}$.  There is a $\mathscr{C}_{\ell}$-decomposition of $K_m-I$, $m$ even,  if and only if $3 \leq \ell \leq m$ and $\ell \mid \frac{m(m-2)}{2}$. 
\end{thm}
A natural next step is to consider $\ell$-cycle decompositions of $\Kmn$.  
Obvious necessary conditions for the existence of such a decomposition are that $\ell$ is at most the number of vertices in $\Kmn$, that the degree $(m-1)n$ is even and that $\ell$ divides the number of edges of $\Kmn$, summarized in the following lemma.
\begin{lemma} \label{necessary}
If there exists a $\mathscr{C}_{\ell}$-decomposition of $\Kmn$, then $3 \leq \ell \leq mn$, $(m-1)n$ is even and $\ell \mid \binom{m}{2}n^2$.
\end{lemma}
These conditions have been shown to be sufficient in several cases.  The results of~\cite{AlspachGavlas, Sajna} show sufficiency when $n \in \{1,2\}$.  Other cases that have been settled include that
$m \leq 5$~\cite{BillingtonCavenaghSmith1, BillingtonCavenaghSmith2, Cavenagh}, $\ell \in \{3,4,5,6,8\}$~\cite{BillingtonHoffmanMaenhaut, CavenaghBillington, CockayneHartnell, Hanani}, and $\ell$ is prime~\cite{ManikandanPaulraja}, twice a prime~\cite{Smith}, or the square of a prime~\cite{Smith-prime, SC3}.  Among the most general results is that the obvious necessary conditions are sufficient if the cycle length $\ell$ is small relative to the number of parts $m$, in particular $\ell\leq m$ if $n$ is odd or $2m$ if $n$ is even~\cite{SC1, SC2}; see also \cite{BS} for some recent work on decompositions into cycles of variable length. Nevertheless, the existence problem for cycle decompositions of the complete multipartite graph remains open in general.    

In this paper, we consider the problem of constructing {\em cyclic} $\ell$-cycle systems of $\Kmn$.  To define this concept, we first recall the definition of a \emph{Cayley graph} on a group $G$ with connection set $\Omega$, denoted by $\Cay[G:\Omega]$.  Let $G$ be an additive group, not necessarily abelian,  and let $\Omega \subseteq G\setminus \{0\}$ such that for every $\omega \in \Omega$ we also have $-\omega \in \Omega$.  The Cayley graph $\Cay[G:\Omega]$ is the graph whose vertices are the elements of $G$ and in which two vertices are adjacent if and only if their difference is an element of $\Omega$ (an analogous definition can be given in multiplicative notation).

Consider the natural action of $G$ on the cycles of $\Gamma=\Cay[G:\Omega]$: given a cycle $C=(c_0,c_1,\ldots,c_{\ell-1})$ in $\Gamma$ and $g \in G$, we define $C+g$ to be the cycle 
$(c_0+g, c_1+g, \ldots, c_{\ell-1}+g)$.  
The subgroup of $G$ consisting of all the elements $g$ such that $C+g=C$ is called the \emph{$G$-stabilizer} of $C$. 
The set $Orb_G(C)=\{C+g\mid g\in G\}$ of all distinct translates of $C$ is called the 
$G$-orbit of $C$.  For $\Gamma=\Cay[G:\Omega]$, a cycle system of $\Gamma$ is said to be {\em regular under the action of $G$}, or {\em $G$-regular},
if it is isomorphic to a cycle system $\cF$ of $\Cay[G:G\setminus N]$, for a suitable subgroup $N$ of order $n$, such that  $C+g \in \cF$
for every $C \in \cF$ and $g \in G$.  In particular, when $G$ is the cyclic group $\Z_n$, a $G$-regular cycle system is called {\em cyclic}.

Clearly, $\Kmn$ is isomorphic to $\Cay[G:G\setminus N]$ where $G$ is a group of order $mn$ and $N$ is any of its subgroups of order $n$. Note that the right cosets of $N$ in $G$ determine the $m$ disjoint parts of $\Kmn$. In this paper, our primary focus is cyclic cycle systems of $\Kmn$, in which case we take $G=\mathbb{Z}_{mn}$ and $N=m\mathbb{Z}_{mn} = \{mx \mid x \in \Z_{mn}\}$.

Cyclic $\ell$-cycle decompositions of $K_m$ (i.e. the case $n=1$) have been extensively studied, and the existence problem has been solved when $m \equiv 1,\ell \pmod{2\ell}$~\cite{BGL, BDF2003, Kotzig, Rosa1, Vietri}, $\ell=m$~\cite{BDF2004}, $\ell \leq 32$~\cite{WuFu}, 
$\ell$ is twice or three times a prime power~\cite{WuFu, Wu2}, or $\ell$ is even and $m>2\ell$~\cite{Wu}.  For $n=2$, the existence problem 
for $\ell$-cycle systems of $K_m[2] \cong K_{2m}-I$ is solved when $m \equiv 1$ (mod $\ell$)~\cite{BGL} or $\ell \mid 2m$~\cite{JM17}.  
Less is known for cyclic $\ell$-cycle systems of $\Kmn$ with $n \geq 3$.  The case $\ell=3$ is solved in~\cite{WangChang}.  
More generally, cyclic $\ell$-cycle decompositions of $\Kmn$ have been studied for $\ell$ odd and $n=\ell$~\cite{BDF2003} and for Hamiltonian cycle systems of $\Kmn$ with $mn$ even~\cite{JM,MPP, PP}.

In this paper, we focus on the existence of cyclic $\ell$-cycle systems of $\Kmn$ when $2\ell \mid (m-1)n$.  This is a natural case to consider, as it means that we may construct cyclic cycle systems in which all cycle orbits are full, i.e.\ the orbit of any cycle has cardinality $mn$.   
Note that when $\ell \geq 3$ and $2\ell \mid (m-1)n$, the conditions of Lemma~\ref{necessary} hold, so that an $\ell$-cycle system of $\Kmn$ may exist.
A  complete solution for cyclic decompositions is known when
$n \in \{1,2\}$, or when $n=\ell$ and both $\ell$ and $m$  are odd.

\begin{thm}[\cite{BDF2003}] \label{complete}
For any integers $\ell \geq 3$  and $m$ such that $2\ell \mid (m-1)$, there is a cyclic $\ell$-cycle system of $K_m$.
\end{thm}

\begin{thm}[\cite{BGL}] \label{part size 2}
If $\ell \mid (m-1)$, then there is a cyclic $\ell$-cycle system of $K_{m}[2]$ if and only if $m \equiv 0$ or $1 \pmod{4}$.
\end{thm}

\begin{thm}[\cite{BDF2003}] \label{part size l}
Let $m, \ell \geq 3$ be odd with $(m,\ell) \neq (3,3)$.  Then there is a 
cyclic $\ell$-cycle system of $K_{m}[\ell]$.
\end{thm}


We will extend these results to the case $n \geq 3$, giving necessary and sufficient conditions for the existence of a cyclic $\ell$-cycle system of $\Kmn$ when $2\ell \mid (m-1)n$.  As in the results above, the main tools are difference methods.
Our main result is the following theorem.

\begin{thm}\label{MainTheorem}
Let $m,\ell \geq 3$ and $n \geq 1$ be integers such that $2\ell \mid (m-1)n$.  There exists a cyclic $\ell$-cycle system of $\Kmn$ if and only if the following conditions hold:
\begin{enumerate}
\item\label{MT:1} If $n \equiv 2$ (mod 4) and $\ell$ is odd, then $m \equiv 0$ or $1$ (mod $4$).
\item\label{MT:2} If $n \equiv 2$ (mod 4) and $\ell \equiv 2$ (mod $4$), then $m \not\equiv 3$ (mod $4$).
\end{enumerate}
\end{thm}
The paper is organized as follows. Section \ref{basics} contains basic observations, definitions and methods: we first explain the necessity of Conditions \ref{MT:1} and \ref{MT:2} of Theorem \ref{MainTheorem} in subsection \ref{necessary}; we then discuss difference families in \ref{diff} and present a recursive construction in \ref{Section:cc} which will be very useful in what follows. In the rest of the paper, we prove the sufficiency of conditions \ref{MT:1} and \ref{MT:2} by explicitly constructing a cycle system in all possible cases: we deal with cycles of even length $\ell$ in Section \ref{even}, while the case $\ell$ odd, which is more complex, is discussed in Sections \ref{odd}--\ref{odd2}. In Section \ref{odd} we outline the proof of the odd case and present some preliminary lemmas, and then treat separately the case $\ell\mid m-1$ in Section \ref{odd1} and $\ell \mid n$ in Section \ref{odd2}. In Section \ref{conclusion} we make some final remarks on what happens if we study regular, rather than cyclic, systems.


\section{Basics}\label{basics}

\subsection{Necessary conditions for cyclic cycle systems}\label{necessary}
If there is a cyclic $\ell$-cycle system of $\Kmn$, then the conditions of Lemma~\ref{necessary} hold.  However, these conditions are not sufficient for the existence of a cyclic cycle system.  In this section, we state further necessary conditions, which reduce to those of  Theorem~\ref{MainTheorem} when $2\ell \mid (m-1)n$, and consider necessary conditions for the existence of regular cycle systems of $\Kmn$ more generally.
We start by recalling a result, proven in \cite{MPP} (see also \cite{BR}), which gives us necessary conditions for the existence of a cyclic $\ell$-cycle systems of $\Kmn$.  Here, given a positive integer $x$, we denote by $|x|_2$ the largest $e$ for which $2^e$ divides $x$.

\begin{thm}[\cite{MPP}]\label{nonexistence}
Let $n$ be an even integer. A cyclic $\ell$-cycle system of $\Kmn$ cannot exist
in each of the following cases:
\begin{itemize}
\item[(a)] $m\equiv 0\pmod 4$ and $|\ell|_2=|m|_2+2|n|_2-1$;
\item[(b)] $m\equiv 1\pmod 4$ and $|\ell|_2=|m-1|_2+2|n|_2-1$;
\item[(c)] $m\equiv 2,3\pmod 4$, $n\equiv 2\pmod 4$ and $\ell \not\equiv 0 \pmod 4$;
\item[(d)] $m\equiv 2,3\pmod 4$, $n\equiv 0\pmod 4$ and $|\ell|_2 = 2|n|_2$.
\end{itemize}
\end{thm}

As we are interested in the case where $2\ell \mid (m-1)n$, we note the following consequence.

\begin{cor} \label{nonexistence corollary}
Suppose $2\ell \mid (m-1)n$.  There does not exist a cyclic $\ell$-cycle system of $\Kmn$ if either of the following holds:
\begin{enumerate}
\item $n\equiv 2\pmod{4}$, $\ell$ is odd and $m\equiv 2,3 \pmod{4}$, or 
\item $n \equiv 2 \pmod{4}$, $\ell\equiv 2\pmod{4}$ and $m\equiv 3\pmod{4}$.
\end{enumerate}
\end{cor}

\subsection{Difference Families}\label{diff}

We now describe the general method we use to construct cyclic $\ell$-cycle systems of $K_m[n]$ in the case where 
$2\ell$ is a divisor of $(m-1)n$.

We will view $\Kmn$ as the Cayley graph $\Cay[\Z_{mn}:\Z_{mn}\setminus m\Z_{mn}]$, 
where by  $m\Z_{mn}$ we mean the only subgroup of order $n$ of $\Z_{mn}$; 
thus vertices of $\Kmn$ will generally be taken as elements of $\Z_{mn}$ 
and the parts of $\Kmn$ as the cosets of $m\Z_{mn}$ in $\Z_{mn}$.

%

Given a cycle $C=(c_0,c_1,\ldots,c_{\ell-1})$ with vertices in  $\Z_{mn}$, 
the multiset $\Delta C~=~\{\pm(c_{h+1}-c_{h})\ |\ 0\leq h < \ell\}$,
where the subscripts are taken modulo $\ell$,
is called the \emph{list of differences} from $C$.
More generally, given a family $\cF$ of cycles with vertices in  $\Z_{mn}$,  
by $\Delta \cF$ we mean the union (counting multiplicities) of all multisets $\Delta C$, where  $C\in \cF$.  
\begin{nota}
We will frequently consider intervals of consecutive differences, and for $a,b\in\Z$ with $a \leq b$, we will use the notation $[a,b]$ to denote the set $\{a, a+1, \ldots, b\}$.  If $a>b$, then $[a,b] = \emptyset$.
\end{nota}

\begin{defini}
An $(mn, n, \sC_\ell)$-\emph{difference family} (DF in short) is a family 
$\cF$ of $\ell$-cycles with vertices in $\Z_{mn}$ such that 
$\Delta \cF = \Z_{mn}\setminus m\Z_{mn}$. 
In other words, an $(mn, n, \sC_\ell)$-\emph{DF} is a set of base cycles whose lists of differences partition between them 
$\Z_{mn}\setminus m\Z_{mn}$.
\end{defini}

Since $|\Delta C|= 2\ell$ for every $C\in \cF$, it follows that
$2\ell |\cF| = |\Z_{mn}\setminus m\Z_{mn}| = (m-1)n$. Therefore, a necessary condition for the existence of an $(mn, n, \sC_\ell)$-\emph{DF} $\cF$ is that $2\ell$ is a divisor of $(m-1)n$, so that $|{\cF}|=(m-1)n/2\ell$. 

Let us recall the following standard result 
(see, for instance, \cite[Proposition 1.2]{BuGio08}).

\begin{prop}\label{DFandCS}  If there exists an $(mn, n, \sC_\ell)$-\emph{DF}, 
then $2\ell$ is a divisor of $(m-1)n$, and there exists 
a cyclic $\ell$-cycle system of $K_m[n]$.
\end{prop}
\begin{proof} Let $\cF=\{C_1, C_2, \ldots, C_t\}$ be
an $(mn, n, \sC_\ell)$-\emph{DF}. 
It is easy to check that $\bigcup_{i=1}^t Orb_{\Z_{mn}}(C_i)$ is 
the desired cyclic $\ell$-cycle system of $K_m[n]$.
\end{proof}

Note that in the cycle system we obtain from the difference family all cycles will have trivial stabilizer, so that all the orbits on the cycles are full orbits.

\begin{ex}
Let $\ell=6$, $m=7$, $n=4$, and let 
$C_1 = (0,-1,2,-4,4,-5)$ and  $C_2 = (0,-2,2,-9,3,-10)$ be two $6$-cycles with vertices in $\Z_{28}$. Since
\[
\Delta C_1 = \pm \{1,3,5,6,8,9\} \mbox{ and } 
\Delta C_2 = \pm \{2,4,10,11,12,13\},
\]
we have that $\Delta C_1\cup\Delta C_2 = 
\pm[1,13]\setminus\{7\} = \Z_{28}\setminus 7\cdot\Z_{28}$, hence
$\cF=\{C_1, C_2\}$ is a $(28, 4, \sC_6)$-DF.
It is not difficult to check that
the set $\{C_1 + j, C_2+ j \mid j\in \Z_{28}\}$ of all translates of $C_1$ and $C_2$ under the action of $\Z_{28}$ is a cyclic $6$-cycle system of $K_{7}[4]$.
\end{ex}

As a consequence of Proposition~\ref{DFandCS}, Corollary~\ref{nonexistence corollary} gives further necessary conditions for the existence of an $(mn,n,\mathscr{C}_{\ell})$-DF.  We thus make the following definition.
\begin{defini}
Let $m, \ell \geq 3$ and $n \geq 1$ be integers.  We call the triple $(mn,n,\ell)$ {\em admissible} if $2\ell \mid (m-1)n$, and the following conditions are both satisfied.
\begin{enumerate}
\item If $n \equiv 2 \pmod{4}$ and $\ell$ is odd, then $m \equiv 0$ or $1\pmod{4}$.
\item If $n \equiv 2 \pmod{4}$ and $\ell \equiv 2 \pmod{4}$, then $m \not\equiv 3 \pmod{4}$.
\end{enumerate}
\end{defini}
Thus if there exists an $(mn,n,\mathscr{C}_{\ell})$-DF, then $(mn,n,\ell)$ is admissible.

We note that the results quoted in Theorems~\ref{complete}, \ref{part size 2} and \ref{part size l} are proved using difference families.  For future reference, we restate these results using the language of difference families.
\begin{thm}[\cite{BDF2003}] \label{complete DF}
For any integers $\ell \geq 3$  and $m$ such that $2\ell \mid (m-1)$, there is an $(m,1,\mathscr{C}_{\ell})$-DF.
\end{thm}

\begin{thm}[\cite{BGL}] \label{part size 2 DF}
If $\ell \mid (m-1)$, then there is a $(2m,2,\mathscr{C}_{\ell})$-DF if and only if $m \equiv 0$ or $1 \pmod{4}$.
\end{thm}

\begin{thm}[\cite{BDF2003}] \label{part size l DF}
Let $m, \ell \geq 3$ be odd with $(m,\ell) \neq (3,3)$.  Then there is a 
$(m\ell,\ell, \mathscr{C}_\ell)$-DF.
\end{thm}

\subsection{A Blow-up Construction} \label{Section:cc}

The following result will be an essential tool in our later constructions to blow up parts in a cyclic cycle system of $K_{m}[n]$ and increase cycle lengths.

\begin{thm}\label{cc} 
If there is an $(mw, w, \sC_\ell)$-DF,
$u$ is a positive divisor of $s>0$, and $\ell(s-1)$ is even, then the following hold:
\begin{enumerate}
  \item there exists an $(mws, ws,\sC_{\ell})$-DF; 
  \item there exists a cyclic $\ell u$-cycle system of $K_{m}[ws]$.
\end{enumerate} 
\end{thm}    
\begin{proof}

Let $\cF$ be an $(mw, w, \sC_\ell)$-\emph{DF},
let $u$ be a positive divisor of $s$ and set  $t=s/u$. 
For every cycle $C$ of $\cF$, with $C=(c_0, c_1, \ldots, c_{\ell-1})$,
and for every $j\in [0,s-1]$, we define the $\ell u$-cycle 
   $C^j=(c^j_0, c^j_1, \ldots, c^j_{\ell u-1})$ as follows:
   \[
    c^j_i = 
    \begin{cases}
        c_i          & \text{if $i$ is even, and $i\leq\ell-2$},\\
        c_i + jmw    & \text{if $i$ is odd, and $i\leq\ell-1$},\\
        c_i + jmw/2  & \text{if $i=\ell-1$ is even}, \\
      c^j_r + qtmw & \text{if $i=q\ell + r$ with $1\leq q\leq u-1$ and $0\leq r<\ell$.}
    \end{cases}
   \]  
We point out that the vertices of $C$ are considered as integers in $[0,mw-1]$,
while the vertices of $C^j$ are elements of $\Z_{mws}$.

We recall that, by assumption, $(s-1)\ell$ is even, hence $s$ is odd when 
$\ell$ is odd. In this case, the map 
$x\in mw\Z_{mws} \mapsto 2x\in mw\Z_{mws}$ is bijective, which means that for every 
$x\in mw\Z_{mws}$ the element $x/2$ is uniquely determined.

Set $\cF' = \{C^j\mid C\in \cF, j\in[0,s-1]\}$. We start showing that
$\Delta \cF'$ contains every element of $\Z_{mws}~\setminus~m\Z_{mws}$.
Let $d = mwj + k \in \Z_{mws}\setminus m\Z_{mws}$, where $j\in[0,s-1]$
and $k\in [0,mw-1]$ is not a multiple of $m$. Recalling that $\cF$ is 
an $(mw, w, \sC_\ell)$-DF, 
there exists a cycle
$C=(c_0, c_1, \ldots, c_{\ell-1})$ of $\cF$ such that 
$c_{i+1} \equiv c_i + k \pmod{mw}$ (replacing $k$ with $-k$ if necessary).
It is not difficult to check that $d\in \Delta C^{h}$ where $h\in[0,s-1]$ is the following:
   \[
    h \underset{\pmod{s}}{\equiv}
    \begin{cases}
       j      & \text{if $i$ is even, and $i\leq\ell-2$},\\
      -j      & \text{if $i$ is odd, and $i\leq\ell-3$},\\
     -2j      & \text{if $i=\ell-2$ is odd}, \\         
      2(t-j)  & \text{if $i=\ell-1$ is even}, \\
      (t-j)    & \text{if $i=\ell-1$ is odd}. \\      
    \end{cases}
   \] 
Hence, $\Delta \cF' \supseteq \Z_{mws}~\setminus~m\Z_{mws}$. 
Since $|\Delta \cF'| = 2\ell u|\cF'|= 2\ell u|\cF|s = (m-1)wsu$, when $u=1$ we have that $\Delta \cF' = \Z_{mws}~\setminus~m\Z_{mws}$, hence $\cF'$ is the desired $(mws, ws, \sC_{\ell})$-DF.

It is left to show that $\cF''= \bigcup_{C\in\cF'} Orb(C)$ 
is a cyclic $\ell u$-cycle system of $K_m[ws]$, where 
$Orb(C)$ denotes the $\Z_{mws}$-orbit of $C$.
We denote by $\epsilon$ the number of edges of $K_m[ws]$ -- counted with their multiplicity -- covered by the cycles in $\cF''$.
By construction, $C + tmw = C$ for every $C\in \cF'$, 
then $|Orb(C)|\leq \frac{mws}{u}$, hence
\begin{equation}\label{upperbound}
\epsilon= \ell u |\cF''| \leq \ell u |\cF'|\frac{mws}{u} = 
\ell u s|\cF|\frac{mws}{u} 
= |E(K_m[ws])|. 
\end{equation}
Therefore, it is enough to show that every edge of $K_{m}[ws]$ lies in 
at least one cycle of $\cF''$. 
By recalling that $\Delta \cF' \supseteq \Z_{mws}~\setminus~m\Z_{mws}$, it follows that every edge $\{x, x+d\}$ of $K_{m}[ws]$ -- hence with $d\not\in m\Z_{ws}$ -- belongs to some translate of the cycle of 
$\cF'$ whose list of differences contains $\pm d$. Therefore,
$\cF''$ is a cyclic $\ell u$-cycle system of $K_m[ws]$.
\end{proof}

Note that some recursive constructions similar to the one above can be found in \cite{BP}.

\begin{ex} Let $m=s=3$ and $\ell=w=5$. Also,
  let $c_0 = 0, c_1 = 1, c_2 = 5, c_3 = 10, c_4 = 8$. Setting $C=(c_0, c_1, c_2, c_3, c_4)$, we have that
  $\Delta C = \pm\{1,2,4,5,8\}$. Hence, if the vertices of $C$ are considered modulo $15$, we obtain
  a $(15, 5, \sC_5)$-DF.
 
  We take $u=1$ and, following the proof of Theorem \ref{cc}, for every $j \in [0, 2]$ 
  we define the $5$-cycle 
  $C^j=(c^j_0, c^j_1, c^j_2, c^j_3, c^j_4)$ as follows:
  \[
    c^j_i = 
    \begin{cases}
        c_i          & \text{if $i=0, 2$},\\
        c_i + 15j    & \text{if $i=1, 3$},\\
        c_i + 30j    & \text{if $i=4$}. \\      
    \end{cases}
  \]    
  Hence $C^0=C$, $C^1 = (0, 16, 5, 25, 38)$, and $C^2 = (0, 31, 5, 40, 23)$. One can check that
  $\cF'=\{C^0, C^1, C^2\}$ is a $(45, 15, \sC_5)$-DF.

  Finally, we take $u=3$, and for every $j \in [0, 2]$ we let 
  $C^j=(c^j_0, c^j_1, \ldots, c^j_{5u-1})$ be the $5u$-cycle 
  defined as follows:
  \[
    c^j_i = 
    \begin{cases}
        c_i            & \text{if $i = 0, 2$},\\
        c_i + 15j      & \text{if $i = 1, 3$},\\
        c_i + 30j      & \text{if $i = 4$},   \\
      c^j_{i-5} + 15  & \text{if $i \in [5, 9]$},\\    
      c^j_{i-10}+ 30  & \text{if $i \in [10, 14]$}.                      
    \end{cases}
  \]
  We then have 
  \begin{align*}
    & C^0 = (0,  1, 5, 10,  8,    15, 16, 20, 25, 23,   30, 31, 35, 40, 38), \\
    & C^1 = (0, 16, 5, 25, 38,    15, 31, 20, 40,  8,   30,  1, 35, 10, 23), \\
    & C^2 = (0, 31, 5, 40, 23,    15,  1, 20, 10, 38,   30, 16, 35, 25,  8),   
  \end{align*}     
  and the set $\cF''=\{C^j + h\mid h\in [0,14]\}$ is a   
  cyclic $15$-cycle system of $K_3[15]$.
\end{ex}

\section{Cycles of even length}\label{even}
In this section we construct cyclic $\ell$-cycle systems of $K_m[n]$
when $\ell$ is an even divisor of $(m-1)n/2$. By Proposition
\ref{DFandCS}, it is enough to provide suitable difference families. 
We will build these difference families by making use of Lemma \ref{lemma1}, which can be thought of as a generalization of Lemma 5.3 in \cite{MT}, proved using alternating sums.

\begin{defini}
If $D=\{d_1, d_2, \ldots, d_{2k}\}$ is a set of positive integers, with $d_{i}< d_{i+1}$ 
for $i\in[1,2k-1]$,
the {\em alternating difference pattern} of $D$ is the sequence $(s_1, s_2, \ldots, s_k)$ where
$s_i=d_{2i}-d_{2i-1}$ for every $i\in[1,k]$. Furthermore, $D$ is said to be {\em balanced} if there exists an integer
$\tau\in[1,k]$ such that
$\sum_{i=1}^{\tau} s_i = \sum_{i=\tau+1}^k s_i$.
\end{defini}

\begin{lemma}\label{lemma1}
If $D$ is a balanced set of $2k$ positive integers, then there exists a 
$2k$-cycle $C$ such that $\Delta C=\pm D$ and 
$V(C)\subset [-d, d']$, where $d=\max D$ and $d'=\max (D\setminus\{d\})$.
\end{lemma}
\begin{proof} 
Let $D=\{d_1, d_2, \ldots, d_{2k}\}$ with $d_i< d_{i+1}$ for $i\in [1, 2k-1]$.
Since $D$ is balanced, 
there is $\tau\in[1,k]$ such that 
$\sigma =\sum_{i=1}^{2\tau} (-1)^i d_i - \sum_{i=2\tau+1}^{2k} (-1)^i d_i = 0$.
Let $\delta_1, \delta_2, \ldots, \delta_{2k}$ be the sequence obtained by reordering the integers in $D$ as follows:
  \[\delta_i = 
    \begin{cases}
      d_i     & \text{if $i\in[1,2\tau]$},\\
      d_{i+1} & \text{if $i\in[2\tau+1, 2k-1]$},\\
      d_{2t+1}  & \text{if $i=2\tau$}.\\            
    \end{cases}  
  \]
Set $c_0=0$ and $c_i=\sum_{h=1}^{i} (-1)^h \delta_{h}$ for $i\in[1, 2k-1]$.
Since $0<\delta_1<\delta_2<\cdots<~\delta_{2k-1}$, we have that $c_i\neq c_j$ whenever $i\neq j$. Also,
the following inequalities hold:
\begin{align*}
  0 &\leq \sum_{h=1}^j (\delta_{2h}-\delta_{2h-1}) = \sum_{h=1}^{2j} (-1)^h \delta_{h}  
    = c_{2j} = -\delta_1 + \sum_{h=1}^{j-1} (\delta_{2h}-\delta_{2h+1}) + \delta_{2j} \leq \delta_{2j}
\end{align*}
  for every $j\in [1, k-1]$, and
\begin{align*}
  -\delta_{2j+1} &\leq \sum_{h=1}^j (\delta_{2h}-\delta_{2h-1}) -\delta_{2j+1} = \sum_{h=1}^{2j+1} (-1)^h \delta_{h}  \\
    &= c_{2j+1} = -\delta_1 + \sum_{h=1}^{j} (\delta_{2h}-\delta_{2h+1})  \leq -\delta_{1}
\end{align*}
  for every $j\in [0, k-1]$. Therefore, every $c_i$ belongs to $[-\delta_{2k-1},\delta_{2k-2}]$
  where  $\delta_{2k-1}=\max D$ and $\delta_{2k-2}=\max (D\setminus\{\delta_{2k-1}\})$.

  To prove that $C=(c_0, c_1, \ldots, c_{2k-1})$ is the desired $2k$-cycle, it is left to show that 
  $\Delta C = \pm D$. Note that
  \begin{align*}
    c_{2k-1}-c_0 & = c_{2k-1} = \sum_{h=1}^{2k-1} (-1)^h \delta_{h} =
      \sum_{h=1}^{2\tau} (-1)^h d_{h} + 
      \sum_{h=2\tau +1}^{2k-1} (-1)^{h} d_{h+1} \\
      & = \sum_{h=1}^{2\tau} (-1)^h d_{h} - 
        \sum_{h=2\tau+2}^{2k} (-1)^{h} d_{h} = \sigma - d_{2\tau+1} = \sigma  - \delta_{2k}.
  \end{align*}
  By recalling that $\sigma=0$, we have that $c_{2k-1}-c_0 = - \delta_{2k}$. 
  Finally, $c_i=c_{i-1}+ (-1)^i\delta_i$ for every $i\in[1,2k-1]$, therefore
  $\Delta C= \pm\{\delta_1, \delta_2, \ldots, \delta_{2k}\} = \pm D$, and this completes the proof.
\end{proof}

\begin{rem}
We note that Lemma~\ref{lemma1} constructs the cycle $C$ with vertices in $\mathbb{Z}$.  In practice, we will use this lemma to construct cycles in $K_m[n]$ with vertices in $\mathbb{Z}_{mn}$; the condition $V(C) \subset[-d,d']$ ensures that $C$ is a cycle provided $mn > d + d'$.
\end{rem}

\begin{ex}
Take $k=6$ and $D=\{1,3,5,7,8,9,10,12,14,15,17,19\}$.
Since the alternating difference pattern of $D$ is $(2,2,1,2,1,2)$, $D$~is clearly balanced.
%
%

Following the notation of Lemma \ref{lemma1}, we have $(\delta_1, \delta_2, \ldots, \delta_{12})=(1,3,5,7,8,9,$ $12,14,15,17,19,10)$, and  the $12$-cycle $C=(0, c_1, \ldots, c_{2k-1})$, built using this sequence, where $c_i=\sum_{h=1}^{i} (-1)^h \delta_{h}$ for 
$i\in[1,11]$ is the following
\[C=(0,-1,2,-3,4,-4,5,-7,7,-8,9,-10).\]
Note that $V(C)\subseteq [-19, 17]$ and  $\Delta C=\pm D$.
\end{ex}

%

\subsection{$\ell\equiv 0 \pmod{4}$}

We first consider the case in which the cycle length $\ell$ is a multiple of 4 and $2\ell \mid (m-1)n$,  hence the nonexistence conditions of Corollary~\ref{nonexistence corollary} are never realized.  
Indeed, we can use Lemma \ref{lemma1} to build an $(mn, n, \sC_\ell)$-DF, thus proving that in this case 
we always have a cyclic $\ell$-cycle system for $\Kmn$.

\begin{thm}\label{4lambda}
If $4\mid \ell$ and $2\ell \mid n(m-1)$, then there exists an $(mn, n, \sC_\ell)$-DF, and hence there exists a cyclic 
$\ell$-cycle system of $\Kmn$.
\end{thm}
\begin{proof} Set $D=[1, \lfloor nm/2\rfloor]\setminus ([1,n]\cdot m)$ and note that
$\pm D = \Z_{mn} \setminus m\Z_{mn}$. 
To build an $(mn, n, \sC_\ell)$-DF, it is enough to 
show that $D$ can be partitioned into a family of balanced $\ell$-sets, and apply
Lemma \ref{lemma1}.   
The existence of a cyclic $\ell$-cycle system of $\Kmn$ then follows  from  Proposition \ref{DFandCS}.

\textbf{Case 1:} $m$ is odd. We recall that by assumption $|D| = (m-1)n/2$ is a multiple of $\ell$, 
hence $(m-1)n/2 = q\ell$ for some $q>0$.
Now, let $D=\{d_1, d_2, \ldots, d_{q\ell}\}$ with $d_i<d_{i+1}$. 
Since $m$ is odd, one can check that $d_{2i}-d_{2i-1}=1$ for every $i\in[1,q\ell/2]$.
Therefore, we can partition $D$ into the subsets 
$D_j=\{d_{\ell j+1}, d_{\ell j+2}, \ldots, d_{\ell (j+1)}\}$ whose alternating difference pattern
is $(1,1,\ldots, 1)$ for every $j\in[0,q-1]$. Since $\ell\equiv 0 \pmod{4}$, every $D_j$ is clearly balanced. 

\textbf{Case 2:} $m$ is even. In this case, $n\equiv 0 \pmod{8}$. 
Let $\ell=4\lambda$,
$n=8t$ for some $t>0$, and let $\iota$ be the involutory permutation of the set $D$ defined by 
$\iota(x)=4tm -x$ for every $x\in D$. 
We notice that 
if $X$ is a subset of $[1,2tm-1]$ with
size $2\lambda$ and 
alternating difference pattern is
$(s_1, s_2, \ldots, s_{\lambda})$, then the set $\overline{X} = X\ \cup \ \iota(X)$ has size 
$4\lambda$ and its alternating difference pattern is 
$(s_1, s_2, \ldots, s_{\lambda}, s_{\lambda}, \ldots, s_2, s_1)$; hence 
$X$ is clearly balanced.

Now, let $A = [1, 2t m-1]\setminus ([1,2t-1]\cdot m)$. 
Recall that by assumption $2\ell \mid n(m-1)$, hence $2\lambda \mid |A|$. 
Let $\{A_1, A_2, \ldots, A_q\}$ be a partition of $A$ into sets of size $2\lambda$ and set $\overline{A_i} = A_i \ \cup\ \iota(A_i)$. As shown above, 
each $\overline{A_i}$ is balanced. Considering that $\{A, \iota(A)\}$ is a partition of $D$, it follows that the $\overline{A}_i$s partition between them $D$ and this completes the proof.

\end{proof}

\begin{ex}
Let $\ell=12$, $m=4$ and $n=16$. Following the notation of the proof of Theorem \ref{4lambda} we have 
$D=[1,32]\setminus ([1,8]\cdot 4)$, and
$A=[1,15]\setminus\{4,8,12\}$.
Setting for instance $A_1=[1,7]\setminus\{4\}$ and $A_2=[9,15]\setminus\{12\}$, we partition $D$ into the two sets $\overline{A_i} = A_i \cup \iota(A_i)$ for $i=1,2$, where
$\iota(A_1) = [25,31]\setminus\{28\}$ and $\iota(A_2)=[17,23]\setminus\{20\}$.
By applying Lemma \ref{lemma1} we build the two cycles 
\begin{align*}
  C_1 &=(0, 1, -1, 2, -3, 3, -4, 22, -5, 24, -6, 25), \\
  C_2 &=(0, 9, -1, 10, -3, 11, -4, 14, -5, 16, -6, 17, 0),
\end{align*}   
such that $\Delta C_i = \pm\overline{A_i}$ for $i=1,2$. Therefore
$\{C_1, C_2\}$ is a $(64, 16, \sC_{12})$-DF.
\end{ex}

\subsection{$\ell\equiv 2 \pmod{4}$}

Let us now consider the case $\ell \equiv 2$ (mod 4).  
We will show that for any such $\ell$, there is a cyclic $\ell$-cycle decompostion of $\Kmn$ whenever the conditions of Theorem~\ref{MainTheorem} hold. 
Our general approach in this case 
is as follows.  Let $\lambda_m=\gcd(m-1,\ell)$ and let $n_0$ be the smallest value for which the triple $(mn_0, n_0, \ell)$ is admissible.  If $\lambda_m \geq 3$, 
we build an $(mn_0, n_0, \sC_{\lambda})$-DF where $\lambda=\lambda_m$ 
(Theorem~\ref{complete DF})
or $2\lambda_m$ 
(Lemma~\ref{Lemma:l=2mod4:2}), and if $\lambda_m \leq 2$, we find an
$(mn_0, n_0, \sC_{\ell})$-DF 
(Lemma~\ref{Lemma:l=2mod4:3}).  
We then obtain a cyclic $C_{\ell}$-decomposition of 
$\Kmn$ by applying Theorem \ref{cc}.

We start by recalling that Theorem \ref{complete DF}  
guarantees the existence of an $(m, 1, \sC_{\ell})$-DF whenever $m\equiv 1 \pmod{2\ell}$
 and $\ell\equiv 2 \pmod{4}$.

We now prove two lemmas which we will need to prove the general existence result.

\begin{lemma}\label{Lemma:l=2mod4:2} There exists a $(4m, 4, \sC_{\ell})$-DF whenever $6\leq \ell \equiv 2 \pmod{4}$ and $\ell\mid 2(m-1)$.
\end{lemma}
\begin{proof}
Let $q=2(m-1)/\ell$, and note that $2q < m-1$; also let 
\[
\mathcal{A} = 
\begin{cases}
[1,2q],                    & \mbox{if $m$ is odd} \\
[1,2q-2] \cup \{m-1,m+1\}, & \mbox{if $m$ is even},
\end{cases}
\]
Since  $q\not\equiv m \pmod{2}$, there exists a partition $\{\{a_i, a_i+2\} \mid i \in [1,q]\}$
of the elements of $A$ into pairs at distance $2$, where $a_q=m-1$ if $m$ is even.
Set
\[
\mathcal{B} = 
\begin{cases}
[2q+1,m-1] \cup [m+1,2m-1],               & \mbox{if $m$ is odd}, \\
[2q-1,m-2] \cup [m+2,2m-2] \cup \{2m+1\}, & \mbox{if $m$ is even},
\end{cases}
\] 
and let $\{B_i \mid i \in [1, q]\}$ be a partition of $\mathcal{B}$ such that each $B_i$ contains $\ell-2$ elements and $\max_{b \in B_i} d < \min_{b \in B_j} d$ whenever $i<j$. Note that each $B_i$ can be partitioned into pairs of consecutive integers except when $i=q$ and $m$ is even. In this case,
$B_q$ can be partitioned into pairs of consecutive integers and a pair at distance three.
Finally, for each $i \in [1,q]$, set 
\[
D_i = \{a_i, a_i+2\} \cup B_i.
\]
Clearly, the $\ell$-sets $D_i$ between them partition $\cA \cup \cB$, and
each $D_i$ has the following alternating difference pattern: 
\[
\begin{cases}
  (2,1,\ldots,1)   & \text{if $i<q$, or $i=q$ and $m$ is odd},\\
  (2,1,\ldots,1,3) & \text{if $i=q\neq 1$ and $m$ is even},   \\
  (\underbrace{1,\ldots,1,2}_{(\ell+2)/4}, 
   \underbrace{1,\ldots,1,3}_{(\ell-2)/4}) & \text{if $i=q= 1$ and $m$ is even}.  
\end{cases}
\]
Therefore, each $D_i$ is balanced and the assertion follows from 
Lemma~\ref{lemma1}.
\end{proof}

\begin{lemma}\label{Lemma:l=2mod4:3}
  There exists an $(mn, n, \sC_{\ell})$-DF whenever $6\leq \ell \equiv 2 \pmod{4}$ 
  and at least one of the following conditions hold:
  \begin{enumerate}
    \item $m\equiv 1 \pmod{4}$ and $2n\equiv 0 \pmod{\ell}$, or
    \item $n\equiv 0\pmod{2\ell}$.
  \end{enumerate}
\end{lemma}
\begin{proof} 
We first consider the case $m\equiv 1 \pmod{4}$ and $2n\equiv 0 \pmod{\ell}$. It is enough to show that there exists an $\left(\frac{m\ell}{2}, \frac{\ell}{2}, \sC_{\ell}\right)$-DF; 
the result then follows from Theorem \ref{cc} with $s=2n/\ell$. 

We have that $q=(m-1)/4$ is the number of cycles in an 
$\left(\frac{m\ell}{2}, \frac{\ell}{2}, \sC_{\ell}\right)$-DF. Also, let 
\[
\mathcal{A} = \left\{ \begin{array}{ll} 
\left[ \frac{(\ell-2)m}{4}+1, \frac{\ell m-2}{4}\right], & \mbox{if $m \equiv 1 \pmod{8}$} \\[2ex]
\left[ \frac{(\ell-2)m}{4}+1, \frac{\ell m-6}{4}\right] \cup \left\{ \frac{\ell m + 2}{4} \right\}, & \mbox{if $m \equiv 5 \pmod{8}$}.
\end{array} \right.
\]
Note that  $\mathcal{A}$ can be partitioned into pairs $\{ \{a_i,a_i+2\} \mid i \in [1,q]\}$. 

Let $\mathcal{B} = [1,(\ell-2)m/4] \setminus m[1,(\ell-2)/4]$, and let 
$\{B_i \mid i \in [1,q]\}$ be a partition of $\mathcal{B}$ such that each $B_i$ contains $\ell-2$ elements and $\max  B_{i} < \min B_{j}$ if $i < j$.  
Since $m\equiv 1 \pmod{4}$, it follows that each $B_i$ can be partitioned into pairs of consecutive integers. 
Now, for each $i \in [1,q]$, set $D_i = \{a_i, a_i+2\} \cup B_i$. 
Clearly, $D_i$ has alternating difference pattern $(1,1,\ldots,1,2)$. Hence
each $D_i$ is balanced, and by Lemma~\ref{lemma1} there exists a set $\cF=\{C_i \mid i\in[1,q]\}$ of $\ell$-cycles with vertices in $\Z_{m\ell/2}$ such that 
$\Delta C_i = \pm D_i$. Since the sets $\pm D_i$ partition between them 
$\pm( \cA \ \cup \cB) = \Z_{m\ell/2}\setminus m\Z_{m\ell/2}$, it follows that
$\cF$ is the desired $\left(\frac{m\ell}{2}, \frac{\ell}{2}, \sC_{\ell}\right)$-DF.

Now suppose $n\equiv 0 \pmod{2\ell}$. It is enough to construct a 
$\left(2\ell m, 2\ell, \sC_{\ell}\right)$-DF and then apply Theorem \ref{cc}
with $s=n/2\ell$. For $i\in[1,m-1]$, let 
$D_i =\{i+jm\mid j\in[0, \ell-2]\ \cup\ \{\ell\}\}$.
Each $D_i$ has alternating difference pattern $(m, \ldots, m, 2m)$;  
hence
$D_i$ is clearly balanced and by Lemma~\ref{lemma1}
there exists a set $\cF=\{C_i \mid i\in[1,q]\}$ of $\ell$-cycles with vertices in 
$\Z_{2\ell m}$ such that 
$\Delta C_i = \pm D_i$. Considering that the sets $\pm D_i$ partition between them 
$\Z_{2\ell m}\setminus m\Z_{2\ell}$, we have that $\cF$ is 
a $\left(2\ell m, 2\ell, \sC_{\ell}\right)$-DF.
\end{proof}

\begin{ex} 
Let $\ell=10$, $m=13$ and $n=5$. 
Following the notation of the proof of Theorem \ref{Lemma:l=2mod4:3}, we have that $q=3$, the set $\mathcal{A} = [27,31] \cup \{33\}$ is partitioned as 
\[
\{ \{27, 29\}, \{28, 30\}, \{31, 33\}\},
\]
and the set $\cB= [1,26] \setminus \{13, 26\}$ is partitioned as follows:
\begin{eqnarray*}
B_1 &=& \{ 1,2,3,4,5,6,7,8\} \\
B_2 &=& \{9, 10, 11, 12, 14, 15, 16, 17\} \\
B_3 &=& \{18, 19, 20, 21, 22, 23, 24, 25\}.
\end{eqnarray*}
Set $D_1= B_1 \cup \{27,29\}$, $D_2= B_2 \cup \{28,30\}$ and 
$D_3 = B_3 \cup \{31,33\}$.  
The cycles of a $(65, 5, \sC_{10})$-DF are given by 
\begin{eqnarray*}
C_1 &=& (0,-1,1,-2,2,-3,3,-5,22,-7) \\
C_2 &=& (0,-9,1,-10,2,-12,3,-14,14,-16) \\
C_3 &=& (0,-18,1,-19,2,-20,3,-22,9,-24).
\end{eqnarray*}
\end{ex}

\begin{ex}
Let $\ell=6$, $m=3$ and $n=2\ell = 12$.
Following the notation of the proof of Theorem \ref{Lemma:l=2mod4:3}, we have that  $q=2$,
\[
D_1 = \{1,4,7,10,13,19\} \mbox{ and }
D_2 = \{2,5,8,11,14,20\}
\]
The cycles of a $(36, 3, \sC_6)$-DF are given by 
\[
C_1 = (0,-1,3,-4,6,-13) \mbox{ and }
C_2 = (0,-2,3,-5,6,-14).
\]
\end{ex}

We now prove the main result of this section, which gives necessary and sufficient conditions for the existence of a cyclic cycle system when $\ell \equiv 2$ (mod~4).

\begin{thm}
Let $\ell, m \geq 3$ and $n \geq 1$ be integers.  If $\ell \equiv 2 \pmod{4}$ and $2 \ell \mid n(m-1)$, then there exists a cyclic $\ell$-cycle system for $K_m[n]$, except when $m \equiv 3 \pmod{4}$ and $n \equiv 2 \pmod{4}$. 
\end{thm}

\begin{proof} When $m \equiv 3 \pmod{4}$ and $n \equiv 2 \pmod{4}$, 
the non-existence of a cyclic $\ell$-cycle system for $K_m[n]$ 
follows from Corollary~\ref{nonexistence corollary}. 

We now show sufficiency. Let $6\leq \ell \equiv 2 \pmod{4}$ such that 
$2 \ell \mid n(m-1)$, and assume that $n \not\equiv 2 \pmod{4}$ when $m \equiv 3 \pmod{4}$. Set $\lambda_m=\gcd(\ell, m-1)$ and note that
$m$ and $\lambda_m$ have different parities,  and $\lambda_m \equiv 2 \pmod{4}$ when 
$m$ is odd.

If $\lambda_m \geq 3$ and $m \equiv 1 \pmod{4}$, then $m \equiv 1 \pmod{2\lambda_m}$.
By Theorem~\ref{complete DF}, there exists an $(m,1,\sC_{\lambda_m})$-DF.
The result then follows by Theorem~\ref{cc}, taking $u=\ell/\lambda_m$ and $s=n$. 
If $\lambda_m \geq 3$ and $m \not\equiv 1 \pmod{4}$, then $4 \mid n$.  
Setting $\lambda = \lambda_m$ if $m \equiv 3 \pmod{4}$ and $\lambda = 2\lambda_m$ otherwise, by Lemma~\ref{Lemma:l=2mod4:2} there exists a 
$(4m,4,\sC_{\lambda})$-DF.
The result then follows by Theorem~\ref{cc}, taking $u=\ell/\lambda$ and $s=n/4$.

Finally, we assume that $\lambda_m \leq 2$. 
If $m \equiv 1 \pmod{4}$, then $\lambda_m=2$, hence $\ell/2$ is a divisor of $n$, 
that is, $2n\equiv 0 \pmod{\ell}$.
If $m \not\equiv 1 \pmod{4}$, then $n\equiv 0 \pmod{2\ell}$.
This is clear when
$\lambda_m=1$. If $\lambda_m=2$, then $m \equiv 3 \pmod{4}$, and by assumption 
$n \not\equiv 2 \pmod{4}$. Recalling that 
$2\ell \mid n(m-1)$, we have that $2\ell \mid n$. The result then follows from
Lemma~\ref{Lemma:l=2mod4:3} and Proposition \ref{DFandCS}.
\end{proof}

\section{Cycles of odd length}\label{odd}
In this section we deal with the existence of $\ell$-cycle systems of $K_m[n]$ when $\ell$ is odd and $2\ell|(m-1)n$; the main result is the following theorem.

%
\begin{thm} \label{odd cycle main}
Let $\ell, m \geq 3$ and $n \geq 1$ be integers.  If $\ell$ is odd and $2 \ell \mid n(m-1)$, then there exists a cyclic $\ell$-cycle system for $K_m[n]$, except when  $m \equiv 2, 3 \pmod{4}$ and $n \equiv 2 \pmod{4}$. 
\end{thm}

We first note that the case $\ell=3$, that is the existence of cyclic triple systems of $K_m[n]$ with no short-orbit cycles, has been settled in \cite{Pe, WangChang}.

\begin{thm}[\cite{Pe, WangChang}]\label{l=3}
There exists an $(mn,n,\mathscr{C}_3)$-DF
if and only if $m>2$, $6\mid (m-1)n$, and $m\equiv 0,1 \pmod{4}$ when $n\equiv 2\pmod{4}$.
\end{thm}

To prove the main result, we first consider in Section \ref{odd1} the case where $\ell>3$ is  a divisor of $m-1$,
and $n\equiv 0 \pmod{4}$, and show the following.
\begin{thm}\label{ell divides m-1}
Let $\ell \geq 5$ be odd, and let $m\geq 3$ and $n\geq1$.  
If $m\equiv 1 \pmod{\ell}$ and $n\equiv 0 \pmod{4}$, then there exists a 
$(mn, n, \sC_\ell)$-DF.
\end{thm}
Then, in Section \ref{odd2} we consider the case where $2\ell \mid n$, and show the following.
\begin{thm}\label{ell divides n}
Let $\ell \geq 5$ be odd, and let $m\geq 3$ and $n\geq1$.  
There exists a $(mn, n, \sC_\ell)$-DF in each of the following cases:
\begin{enumerate}
  \item $n=2\ell$ and $m\equiv 0,1\pmod{4}$,
  \item $n\equiv 0 \pmod{4\ell}$.
\end{enumerate}
\end{thm}

We now have all the ingredients we need to prove Theorem \ref{odd cycle main}.
\begin{proof}[Proof of Theorem \ref{odd cycle main}] 
The case $\ell=3$ is dealt with in Theorem \ref{l=3}, so we assume $\ell \geq 5$.  Necessity of the condition that 
$n \not\equiv 2 \pmod{4}$ when $m \equiv 2$ or $3 \pmod{4}$ follows from Corollary~\ref{nonexistence corollary}, 
so we show sufficiency.  

Let $\lambda_m=\gcd(\ell,m-1)$, $\lambda_n=\ell/\lambda_m$, and $n=2^a \lambda_n n'$ where $a \geq 0$ and $n'$ is odd.   Note that if $a=0$, then the condition $2\ell \mid (m-1)n$ implies that $m$ is odd. 

First, suppose that $\lambda_m \geq 3$.  In this case, Theorems~\ref{complete DF} and \ref{part size 2 DF}
(when $a=0,1$), and Theorem~\ref{ell divides m-1} (when $a>1$) guarantee that there is an
$(m2^a,2^a,\mathscr{C}_{\lambda_m})$-DF, and the result follows by applying Theorem~\ref{cc} with $u=\lambda_n$ and $s=\lambda_n n'$.  

Otherwise, $\lambda_m=1$ so that $\ell \mid n$, and by Theorem~\ref{part size l DF} (when $a=0$) and 
Theorem~\ref{ell divides n} (when $a>0$) there exists a
$(m2^a\ell,2^a\ell,\mathscr{C}_{\ell})$-DF.
The result now follows by applying Theorem~\ref{cc} with $u=1$ and $s=n'$.
\end{proof}

We end this section with two lemmas which will be used to construct the difference families of
Theorems \ref{ell divides m-1} and  \ref{ell divides n}.

\begin{lemma}\label{lemma:odd1} Let $D=\{d, d^*\}\ \cup\ \mathbb{X}$ be a set of $2\lambda$ positive integers with $d<d^*$. If $\mathbb{X}$ can be partitioned into pairs of consecutive integers,
  then there exists a path 
$P=0, p_1, p_2, \ldots, p_{2\lambda}$ 
of length $2\lambda$ satisfying the following properties:
\renewcommand{\theenumi}{\roman{enumi}}
\begin{enumerate}
  \item \label{prop1} $(p_1, p_2) = (-d, d^*-d)$, and 
  $p_i\in [d^*-d+1, d^*-d + \max \mathbb{X}]$ for $i>2$,
  \item \label{prop2} $p_{2\lambda} = d^* - d  + \lambda-1$,
  \item  $\Delta P = \pm D$.
\end{enumerate}
\end{lemma}
\begin{proof} 
Letting $\mathbb{X}=\{x_1, x_2, \ldots, x_{2\lambda-2}\}$, we can assume that
\begin{equation}\label{X}
   x_i>x_{i+1} \;\;\; \text{and}\;\;\; x_{2j-1} - x_{2j}=1,
\end{equation}  
for every $i\in[1,2\lambda-3]$ and $j\in[1,\lambda-1]$.
Now, let $P=0, p_1, p_2, \ldots, p_{2\lambda}$ be the trail defined as follows:
\[
p_i = d^* - d +
\begin{cases}
  i/2-1 & \text{if $i\in[1,2\lambda]$ and $i$ is even},\\
  -d^*  & \text{if $i=1$}, \\
  x_{i-2} + (i-3)/2 & \text{if $i\in[3,2\lambda]$ and $i$ is odd},   
\end{cases}
\]
By property \eqref{X}, it is not difficult to check that the sequence 
$p_1, 0, p_2, p_4,$ $\ldots, p_{2\lambda}, p_{2\lambda-1}, p_{2\lambda-3}, \ldots, p_3$ is strictly increasing. Therefore, $P$ is a path, and for every $i>2$, we have that $p_i\in [p_4, p_3] = [d^*-d+1, d^*-d + x_1]$ where $x_1=\max \mathbb{X}$.
Also, 
\begin{align*}
\Delta P &= \pm\{d, d^*\}\ \cup\ \pm\{p_{2j+1}-p_{2j}, p_{2j+1}-p_{2j+2}\mid j\in[1, \lambda-1]\} \\
         &= \pm\{d, d^*\}\ \cup\ \pm\{x_{2j-1}, x_{2j-1} -1 \mid j\in[1, \lambda-1]\} = \pm D.
\end{align*}
Therefore, $P$ is the desired path.
\end{proof}

\begin{ex}
Let $\lambda=3$ $d=9, d^*=11$ and $\mathbb{X}=\{7,8,13,14\}$: the path is $(0,-9,2,16,3,11,4)$
\end{ex}

\begin{nota}
We will use the notation $[a,b]_e$ (resp. $[a,b]_o$) to denote the set of even (resp. odd) integers in $\{a, a+1, \ldots, b\}$.
Also,
given nonempty sets $X_i\subseteq \Z$ and integers $c_i, c_i'$, for $i\in[1,t]$,  
we denote by 
$\sum _{i=1}^t c_i\cdot X_i\cdot c'_i$
the subset of $\Z$ defined as follows: 
\[\sum _{i=1}^t c_i\cdot X_i\cdot c'_i = 
\left\{\sum _{i=1}^t c_ix_ic'_i \mid x_i\in X_i\;\text{for every $i\in[1,t]$}\right\}.\] 
If some $X_i = \emptyset$, then we define $\displaystyle\sum _{i=1}^t c_i\cdot X_i\cdot c'_i = \emptyset$.
\end{nota}

In the proofs of Theorems~\ref{ell divides m-1} and~\ref{ell divides n}, a crucial ingredient will be the following Lemma \ref{lem:AA*}. 

\begin{lemma}\label{lem:AA*} 
Let $I$ and $J$ be two non-empty intervals of $\Z$, with $|I|<\mu$, and 
 set $A = I + J\cdot \mu$. For every $\tau\in \Z$, there is a bijection
 $a\in A \mapsto a^* \in A+\tau$ such that
 \[
   \{a^* - a\mid a\in A\} =  
   \Big(\big[1, 2|I|\big]_o + \big[1, 2|J|\big]_o\cdot \mu\Big) + \tau - |I| - |J|\mu.
 \]  
\end{lemma}
\begin{proof} It is not difficult to check that the map
$a\in A \mapsto a^* \in A+\tau$, with $a^*= \max A + \min A + \tau - a$, is 
a bijection.
%
%

Let $I= [i_1 -s_I +1, i_1]$ and $J=[j_1 -s_J +1, j_1]$ be intervals of size $s_I$ and $s_J$, respectively. For every $a= i + j\mu\in A$, we have that
$a^* = (2i_1 - s_I +1 -i)+ (2j_1 - s_J +1 -j)\mu +\tau$,
hence 
\begin{align*}
a^* -a &=   2(i_1-i) +1 + (2(j_1 -j) +1)\mu + (\tau - s_I - s_J\mu)\\
       &\in \left(2\cdot\big[0, s_I - 1\big] +1\right) + \left(2\cdot\big[0, s_J - 1\big] +1 \right) \mu
       +(\tau - s_I - s_J\mu)      \\
       &= \big[1, 2s_I\big]_o + \big[1, 2s_J\big]_o\cdot \mu + (\tau - s_I - s_J\mu).
\end{align*}
Since the map $a\mapsto a^*-a$ is injective, the assertion follows.
\end{proof}

\section{The proof of Theorem \ref{ell divides m-1}}\label{odd1}

The aim of this section is to prove Theorem \ref{ell divides m-1}.
The case $n\equiv 4\pmod{8}$ is treated in Proposition \ref{part size 4 nu}, while the case $n\equiv 0 \pmod{8}$ is dealt with in Proposition \ref{part size 8 nu - 1 mod 2ell} for $m$ odd, and in Proposition \ref{part size 8 nu - ell+1 mod 2ell}, for $m$ even.

The idea beneath the three proofs is similar: 
we partition the set $D= [1,mn/2]\setminus [1, n/2]\cdot m$ of differences to be realized into various sets. 
A first set $A$ of size $q$, the cardinality of the DF, will serve as the set of indices for the cycles in the DF, and it will be paired up with a second $q$-set, the set $A^*$. To each pair of elements $(a,a^*)\in A\times A^*$ we will associate a set $X_a\subset D$ of size $\ell-5$ that can be partitioned into pairs of consecutive integers, so that  we can have a path $P_a$ of length $\ell-3$, built using Lemma \ref{lemma:odd1} for each $a \in A$, and whose lists of differences between them cover $D'= (\cup_{a\in A} X_a) \cup A\cup A^* $.
We obtain an $\ell$-cycle $C_a$ by joining the path $P_a$ to a path $Q_a$ of length 3, built to ensure that the differences coming from $Q_a, a \in A$, will describe the set $D\setminus D'$. The set $\cC=\{C_a\mid a\in A\}$
will be the desired difference family. The partitions of $D$ just outlined are given in Lemmas 
\ref{lemma:5subsets:0}, \ref{lemma:5subsets:1} and \ref{lemma:5subsets:2}.

\begin{lemma}\label{lemma:5subsets:0} 
Let $\ell=2\lambda+3 \geq 5$ be odd, let $m\equiv 1\pmod{\ell}$ and set $s=2(m-1)/\ell$. 
Then there exists a partition of $D=[1, 2m-1]\setminus \{m\}$ into five subsets
$A, A^*, B, X$, and $Y$ satisfying the following properties:
\begin{enumerate}
  \item $|A|=|A^*|=|B|= s$, $|X|= (\ell-5)s$,  $|Y| = 2s$;
  \item there is a bijection $a\in A\mapsto a^* \in A^*$ such that 
  $B-\lambda = \{a^*-a\mid a\in A\}$;
  \item $X$ and $Y$ can be partitioned into pairs of consecutive integers.   
  \item $1\not\in A$, $B-\lambda\subset[1, s+1]$, and $Y\subset [1, m-1]$ when $\ell\geq 7$.
\end{enumerate}
\end{lemma}
\begin{proof}
 Let $\epsilon=0$ or $1$ according to whether $\lambda$ is even or odd.
 Also, let $A= A_0\ \cup\ A_1$ and $A^*=(A_0 + \tau_0)\ \cup\ (A_1 + \tau_1)$, 
where $A_h$ and $\tau_h$ are the following: 
 \[
   \begin{array}{c|c|c}
     \hline \rule{0pt}{1\normalbaselineskip}
         &   A_h                & \tau_h  \\[0.5ex] 
         \hline \hline \rule{0pt}{1\normalbaselineskip}
     h=0 &   [-s/2, -1] + m      &  s/2 +1 \\[0.5ex]  
         \hline \rule{0pt}{1\normalbaselineskip}
     h=1 &   [s/2 +1, s] + m  & s/2 + 2\epsilon \\[0.5ex]  \hline 
   \end{array}
 \]
 and set $B=[\lambda  + \epsilon + 1, \lambda + \epsilon + s]$.
It is easy to check that $A_0$ and $A_1$ are disjoint, as are $A_0+\tau_0$ and $A_1 + \tau_1$; hence $|A|=|A^*|=|B|=s$.  We need to show that the sets $A$, $A^*$ and $B$ are pairwise disjoint.  It is straightforward to see that $A \cap A^*=\emptyset$.  To check that $B$ is disjoint from $A \cup A^*$, note that the elements of $A \cup A^*$ are contained in the interval $[m-\frac{s}{2}, m+\frac{3s}{2}+2\epsilon]$.  Thus, it suffices to show that $\lambda+\epsilon+s < m-\frac{s}{2}$, or equivalently, that $\lambda+\epsilon+\frac{3s}{2} < m$.  If $\ell=5$, 
\[
\lambda+\epsilon+\frac{3s}{2} = 1 + 1 + \frac{3(m-1)}{5} = \frac{3m+7}{5} < m
\]
since $m \geq \ell+1=6$.  For $\ell \geq 7$, since $\ell \leq m-1 < m$, we have
\begin{eqnarray*}
\lambda+\epsilon+\frac{3s}{2} &=& \frac{\ell-3}{2} + \epsilon + \frac{3(m-1)}{\ell} \\
& < & \frac{m-3}{2} + 1 + \frac{3(m-1)}{7} \\
&=& \frac{13(m-1)}{14} \\
&<& m
\end{eqnarray*}

 By Lemma \ref{lem:AA*} (with $I=[-\frac{s}{2},-1]$ or $[\frac{s}{2}+1,s]$, $J=\{1\}$ and $\mu=m$), there are bijections
 $a\in A_h \mapsto a^* \in A_h+\tau_h$, $h \in \{0,1\}$, such that
 \begin{align*}
   \{a^* - a\mid a\in A_0\} & 
     = \Big(\big[1, s\big]_o + m\Big) -m  +1
     = \big[1,s\big]_e;\\
   \{a^* - a\mid a\in A_1\} & 
     = \Big(\big[1, s\big]_o + m\Big) -m +2\epsilon
     = \big[1, s\big]_o+2\epsilon.
 \end{align*}  
 Therefore, $\{a^* - a\mid a\in A\} = 
 \big[\epsilon+1, \epsilon +s\big] = B-\lambda$. 

Set $W=D\setminus (A\ \cup\ A^*\ \cup\ B)$. Note that $|W| = (\ell-3)s$,
and both $W_1 = W \ \cap\ [1,m-1]$ and $W \ \cap\ [m+1,2m-1]$ 
are the disjoint union of intervals of even size. In particular, if $\ell\geq 7$ then
$|W_1| = m-1 -3s/2 = (\ell-3)s/2 \geq 2s$.
Therefore, $W$ can be seen as the disjoint union of two subsets $X$ and $Y$ each of which can be partitioned into pairs of consecutive integers, with $|X| = (\ell-5)s$, $|Y|=2s$, and $Y\subset W_1 \subset [1,m-1]$.
Therefore, the sets $A, A^*, B, X, Y$ provide the desired partition of $[1, 2m-1]\setminus\{m\}$.
\end{proof}

\begin{prop} \label{part size 4 nu}
Let $\ell \geq 5$ be odd, and let $m\equiv 1\pmod{\ell}$. 
Then there is a $(4m\nu, 4\nu, \sC_{\ell})$-DF for every odd $\nu\geq 1$.
\end{prop}
\begin{proof} By Theorem~\ref{cc}, it is enough to prove the assertion when $\nu=1$. 

First, if $\ell=m-1=5$, take  $\mathcal{C}=\{(0,11,1,10,2), (0,7,2,5,1)\}$. Since 
$\Delta \mathcal{C} = \Z_{24}\setminus\{0,6,12,18\}$, then $\mathcal{C}$
is a $(24,4,C_5)$-DF. We can therefore 
assume that $(\ell,m)\neq (5,6)$. 

As in Lemma~\ref{lemma:5subsets:0}, let $\lambda=(\ell-3)/2$ and $s=2(m-1)/\ell$.  
By that lemma, there is a partition of
$D=[1, 2m-1]\setminus \{m\}$ into five subsets
$A$, $A^*$, $B$, $X$ and  $Y$ which satisfy the following conditions:
\begin{enumerate}
  \item $|A|=|A^*|=|B|= s$, $|X|= (\ell-5)s$,  $|Y| = 2s$;
  \item there is a bijection $a\in A\mapsto a^* \in A^*$ such that 
  $B-\lambda = \{a^*-a\mid a\in A\}$;
  \item $X$ and $Y$ can be partitioned into pairs of consecutive integers.   
  \item $1\not\in A$, $B-\lambda\subset[1, s+1]$, and $Y\subset [1, m-1]$ when $\ell\geq 7$.
\end{enumerate}
In particular, $X$ can be seen as the disjoint union of $s$ sets $X_a$ of size $\ell-5$, 
indexed over the elements of $A$, 
each of which can be partitioned into pairs of consecutive integers, and
$Y = \{y_{a}, y_{a} -1 \mid a\in A\}$.

We will construct a set $\cal C$ of $s=2(m-1)/\ell$ base cycles, indexed over the elements of $A$, and each obtained as a union of two paths of length $\ell-3$ and $3$. 
By applying Lemma \ref{lemma:odd1} 
(with $d=a \in A$ and $\mathbb{X}=X_a$), we construct the  path $P_a$ of length $2\lambda = \ell-3$
such that
\begin{eqnarray}
  & \label{0Pa:ends}&\text{the ends of $P_a$ are $0$ and $p_{a}= a^*- a + \lambda-1$,}\\
  & \label{0Pa:vertices}&\text{$V(P_a)\subseteq \{0,-a\}\ \cup\ [a^*-a, a^*-a + \max X_a]$}, \\
  & \label{0Pa:delta}&\text{$\Delta P_a = \pm \{a, a^*\}\ \cup\ \pm X_a$,}
\end{eqnarray}
where $\max X_a=0$ when $\ell=5$.
For $a\in A$, let $C_a$ be the closed trail obtained by joining $P_a$ and 
the $3$-path $Q_a = 0,\; -y_a,\; -1,\; p_a$,
and considering its vertices as elements of $\mathbb{Z}_{4m}$. 
We claim that $\mathcal{C} = \{C_a\mid a\in A\}$, is the desired difference family.

We first show that $\Delta \mathcal{C} = \pm D$. 
Recalling \eqref{0Pa:delta} and that $B =\{a^* -a +\lambda \mid a\in A\} = \{p_a+1\mid a\in A\}$,
and considering  that $\Delta Q_a = \pm \{y_{a}, y_{a}-1, p_a+1\}$,
then 
\begin{align*}
  \Delta \mathcal{C} &= \bigcup_{a\in A} \Delta C_a =
  \bigcup_{a\in A} (\Delta P_a\ \cup\ \Delta Q_a)  
   =\pm \bigcup_{a\in A} 
   \left(X_a \ \cup\ \{a, a^*, y_{a}, y_{a}-1, p_a+1\}\right)\\
   & =\pm (D\setminus B) \ \cup\ \pm\{p_a+1\mid a\in A\} = \pm D.
\end{align*}

It is left to show that each $C_a$ is a cycle.  Since $a^*-a\in B-\lambda\subset [1, s+1]$ where $s=2(m-1)/\ell < m$, and
$\max X_a<2m$,  
it follows by \eqref{0Pa:vertices} that 
\[
V(P_a) \subseteq \{-a\}\ \cup\ [0,2m+s] \subseteq \{-a\} \cup [0, 3m-1]. 
\]
But $y_a \in Y \subset [1,m-1]$, so it follows that $P_a$ and $Q_a$ share a vertex other than $0$ or $p_a$ modulo $4m$ if and only if $-a \in \{-1, -y_a\}$. Recalling that $1\not\in A$ and $A\ \cap\ Y$ is empty, we see that the latter condition is not satisfied; thus, $P_a$ and $Q_a$ only share their end-vertices modulo $4m$.
Hence $C_a$ is a cycle for every $a\in A$,
and this completes the proof.
\end{proof}

\begin{ex}
Let $\ell=9$ and $m=10$; we have to build a 
$(40, 4, \sC_{9})$-DF, so we need $q=2$ base cycles. Here $\lambda=3$ and, following 
the proof of Lemma \ref{lemma:5subsets:0}, we have that
\[A_0=\{9\}, A_1=\{12\}, A^*_0=\{11\}, A^*_1=\{15\}, B=\{5,6\},
\] 
and  the index set is $A=\{9,12\}$. Also, 
we can take $Y=[1,4]$, $X_9=\{7,8,13,14\}$ and $X_{12}=[16,19]$. 

The path $P_9$ is the path  $0,-9,2,16,3,11,4$, and we might take $y_9=2$ so that $Q_9= 0,-2,-1,4$, and 
\[
C_9 = P_9\ \cup\ Q_9 = (0,-9,2,16,3,11,4,-1,-2),
\] 
while $P_{12}$ is the path $0,-12,3,22,4,21,5$ and $Q_{12}$ is $0,-4,-1,5$ so that 
\[
 C_{12} = P_{12}\ \cup\ Q_{12} = (0,-12,3,22,4,21,5,-1,-4).
\]
It is easily checked that $\Delta C_9 \cup \Delta C_{12}=\Z_{40}\setminus{10\cdot\Z_{40}}$.
\end{ex}

\begin{lemma}\label{lemma:5subsets:1} 
Let $\ell=2\lambda+3 \geq 5$ be odd, let $m\equiv 1\pmod{2\ell}$ and set $s=4(m-1)/\ell$. 
Then for every integer $\nu \geq 1$, there exists a partition of $D=[1, 4m\nu]\setminus ([1,4\nu]\cdot m)$ into five subsets
$A, A^*, B, X$, and $Y$ satisfying the following properties:
\begin{enumerate}
  \item $|A|=|A^*|=|B|=\nu s$, $|X|= (\ell-5)\nu s$,  $|Y| = 2\nu s$;
  \item there is a bijection $a\in A\mapsto a^* \in A^*$ such that 
  $B-\lambda = \{a^*-a\mid a\in A\}$;
  \item $X$ and $Y$ can be partitioned into pairs of consecutive integers.   
  \item $1\not\in A$, $B-\lambda\subset[1,(2\nu-1)m]$, and $Y\subset [1, (2\nu+1)m-1]$ when $\ell\geq 7$.
\end{enumerate}
\end{lemma}
\begin{proof}
 Let $\epsilon=0$ or $1$ according to whether $\lambda$ is even or odd,
  and set $q=s\nu $.
 
 We start by defining intervals $I_h, J_h$ and integers $\tau_{h}$, for $h\in\{0,1,2\}$, 
 as follows:
 \[
   \begin{array}{c|c|c|c}
     \hline \rule{0pt}{1\normalbaselineskip}
         &   I_h       & J_h               & \tau_h  \\[0.5ex] 
         \hline \hline \rule{0pt}{1\normalbaselineskip}
     h=0 &   [1, s/2]  & [2\nu-1, 3\nu-2]  & (\nu-1)m +s/2 +2\epsilon \\[0.5ex]  
         \hline \rule{0pt}{1\normalbaselineskip}
     h=1 & \multirow{2}{4.5em}{$[-s/2, -1]$}  & [2\nu, 3\nu-2]    & {\nu}m   +s/2 +1 \\ [0.5ex] 
         \cline{1-1}  \cline{3-4} \rule{0pt}{1\normalbaselineskip} 
      
     h=2 &             & \{4\nu-1\}        & s/2 +1    \\[0.5ex]  \hline 
   \end{array}
 \]
 For every $h\in\{0,1,2\}$, set $A_{h} = I_h + J_h\cdot m$, 
 $A_{h}^* = A_{h} + \tau_h$, 
 and let
 \[
 A = \bigcup_{h=0}^{2} A_h \;\;\;\mbox{and}\;\;\; 
 A^* = \bigcup_{h=0}^{2} A^*_h.
 \]
 Also, set $B=[\lambda + \epsilon + 1, \lambda + \epsilon + s] + [0, \nu-1]\cdot 2m$.
 It is not difficult to check that the sets $A_0$, $A_1$, $A_2$, $A^*_{0}$, $A^*_{1}$, $A^*_2$ and $B$ are pairwise disjoint;
 hence $|A|=|A^*|=|B|=\nu s$.
 By Lemma \ref{lem:AA*}, there is a bijection
 $a\in A_h \mapsto a^* \in A^*_h$  such that
 \begin{align*}
   \{a^* - a\mid a\in A_0\} & 
     = \Big(\big[1, s\big]_o + \big[1, 2\nu\big]_o   \cdot m\Big) -m + 2\epsilon 
     = \big[1+ 2\epsilon, s+ 2\epsilon\big]_o + \big[0, 2\nu-2\big]_e \cdot m;\\
   \{a^* - a\mid a\in A_1\} & 
     = \Big(\big[1, s\big]_o + \big[1, 2\nu-2\big]_o \cdot m\Big) + m +1
     = \big[1, s\big]_e + \big[1, 2\nu-2\big]_e \cdot m;\\
   \{a^* - a\mid a\in A_2\} & 
     = \Big(\big[1, s\big]_o + m\Big) -m + 1 = \big[1, s\big]_e.
 \end{align*}  
 Therefore, $\{a^* - a\mid a\in A\} = 
 \big[\epsilon+1, \epsilon +s\big] + \big[0, \nu-1\big] \cdot 2m = B-\lambda$. 

Now set $W=D\setminus (A\ \cup\ A^*\ \cup\ B)$ and note that $|W| = (\ell-3)s\nu$.
Also, for every $j\in[0, 4\nu-1]$ we have that $([1,m-1]+jm)\ \cap\ W$ is the disjoint union of intervals of even size. 
%
%
Therefore, $W$ can be seen as the disjoint union of two subsets $X$ and $Y$, each of which can be partitioned into pairs of consecutive integers,
with $|X| = (\ell-5)\nu s$ and $|Y| = 2s\nu$.
Since   
$\big|(A\cup A^*\cup B) \cap [1, (2\nu+1)m]\big| = (\nu+2)s$, for every $\ell\geq 7$ we have that 
\begin{align*}
  \big|W \cap [1, (2\nu+1)m]\big| &= (2\nu+1)(m-1) - (\nu+2)s = \ell s(2\nu+1)/4 - (\nu+2)s \\
  &= (\ell-2)s\nu/2 +(\ell-8)s/4 \geq 2s\nu + s\nu/2 -s/4 \geq   2s\nu.
\end{align*}
Therefore, without loss of generality, we can assume that
$Y\subset[1, (2\nu+1)m-1]$ when $\ell\geq 7$, and this completes the proof.
\end{proof}

\begin{prop} \label{part size 8 nu - 1 mod 2ell}
Let $\ell \geq 5$ be odd, and let $m\equiv 1\pmod{2\ell}$. 
Then there is a $(8m\nu, 8\nu, \sC_{\ell})$-DF for every $\nu\geq 1$.
\end{prop}
\begin{proof}
Set $\lambda=(\ell-3)/2$ and let $\epsilon\in\{0,1\}$, with $\epsilon\equiv \lambda \pmod{2}$. Also,
set $s=4(m-1)/\ell$ and note that the number of required base cycles is $q={\nu} s$. 

By Lemma \ref{lemma:5subsets:1}, there is a partition of
$D=[1, 4m\nu]\setminus ([1,4\nu]\cdot m)$ into five subsets
$A$, $A^*$, $B$, $X$ and  $Y$ which satisfy the following conditions:
\begin{enumerate}
  \item $|A|=|A^*|=|B|=\nu s$, $|X|= (\ell-5)\nu s$,  $|Y| = 2\nu s$;
  \item there is a bijection $a\in A\mapsto a^* \in A^*$ such that 
  $B-\lambda = \{a^*-a\mid a\in A\}$;
  \item $X$ and $Y$ can be partitioned into pairs of consecutive integers;  
  \item $1\not\in A$, $B-\lambda\subset[1,(2\nu-1)m]$, and $Y\subset [1, (2\nu+1)m-1]$ when $\ell\geq 7$.
\end{enumerate}
In particular, $X$ can be seen as the disjoint union of $q$ sets $X_a$ of size $\ell-5$, indexed over the elements of $A$, 
each of which can be partitioned into pairs of consecutive integers, and
$Y = \{y_{a}, y_{a} -1 \mid a\in A\}$.
By applying Lemma \ref{lemma:odd1} 
(with $d=a \in A$ and $\mathbb{X}=X_a$), we construct a path $P_a$ of length $2\lambda = \ell-3$
such that
\begin{eqnarray}
  & \label{Pa:ends}&\text{the ends of $P_a$ are $0$ and $p_a= a^*- a + \lambda-1$,}\\
  & \label{Pa:vertices}&\text{$V(P_a)\subseteq \{0,-a\}\ \cup\ [a^*-a, a^*-a + \max X_a]$}, \\
  & \label{Pa:delta}&\text{$\Delta P_a = \pm \{a, a^*\}\ \cup\ \pm X_a$},
\end{eqnarray}
where $\max X_a=0$ when $\ell=5$. 
For $a\in A$, let $C_a$ be the closed trail obtained by joining $P_a$ and 
the $3$-path $Q_a = 0,\; -y_a,\; -1,\; p_a$,
and considering its vertices as elements of $\mathbb{Z}_{8m\nu}$. 
We claim that $\mathcal{C} = \{C_a\mid a\in A\}$, is the desired difference family. 

We first show that $\Delta \mathcal{C} = \pm D$. 
Recalling \eqref{Pa:delta} and that $B =\{a^* -a +\lambda \mid a\in A\} = \{p_a+1\mid a\in A\}$,
and considering  that $\Delta Q_a = \pm \{y_{a}, y_{a}-1, p_a+1\}$,
it follows that 
\begin{align*}
  \Delta \mathcal{C} &= \bigcup_{a\in A} \Delta C_a =
  \bigcup_{a\in A} (\Delta P_a\ \cup\ \Delta Q_a)  
   =\pm \bigcup_{a\in A} 
   \left(X_a \ \cup\ \{a, a^*, y_{a}, y_{a}-1, p_a+1\}\right)\\
   & =\pm (D\setminus B) \ \cup\ \pm\{p_a+1\mid a\in A\} = \pm D.
\end{align*}

It is left to show that each $C_a$ is a cycle. 
By \eqref{Pa:vertices}, if $\ell=5$, then  
$V(P_a) = \{0, a, p_a=a^*-a\}$. By recalling that 
$A,Y\subset [1, 4m\nu -1]$, it follows that 
$a\not\in\{-y_a, -1\}$, hence $C_a$ is a cycle. 
Again by \eqref{Pa:vertices}, if $\ell\geq7$, then
$V(P_a) \subseteq \{0,-a\}\ \cup\ [a^*-a, a^*-a + \max X_a]$. 
Recalling conditions 2 and 4, we have that
$a^*-a\in B-\lambda\subset[1, (2\nu-1)m]$ for every $a\in A$. Since
$\max X_a < 4 m \nu$,
then $V(P_a) \subseteq \{0,-a\} \ \cup\ [1, (6\nu-1)m-1]$ for every $a\in A$. 
Recalling that $1\not\in A$ and $Y\subset [1, (2\nu+1)m-1]$ (condition 4), and that $A\cap Y$ is empty, it follows that $\{-1, -y_a\}\not\in V(P_a)$, that is,
$P_a$ and $Q_a$ only share their end-vertices. 
Hence $C_a$ is a cycle for every $a\in A$,
and this completes the proof.
\end{proof} 

\begin{ex}
Let $\ell=9$, $m=19$ and $n=8$, so that $\nu=1, \lambda=3, \epsilon=1, s=8$, 
and the size of our $(152, 8, \sC_{9})$-DF will be $q=\nu s = 8$.
Following the proof of Lemma \ref{lemma:5subsets:1},  we have that
\[A_0=[20,23], A_2=[53,56], A^*_0=[26,29], A^*_2=[58,61], A_1 = A_1^* = \varnothing,B=[5,12],
\]
and our index set is $A=[20,23]\cup [53,56]$. 

We can take for instance 
\[
Y=[1,4]\ \cup\ [13,18]\ \cup\ [24,25]\ \cup\ [30,33],
\] 
and the remaining differences in $D=[1, 152]\setminus \big([1,8]\cdot 19\big)$ can be partitioned to form the eight 4-sets $X_a, a\in A$:
\begin{align*}
  & X_{20} =[34,37], X_{21} =[39,42], X_{22} =[43,46], X_{23} =[47,50],\\
  & X_{53} = \{51,52, 62,63\}, X_{54} =[64,67], X_{55} =[68,71], X_{56} =[72,75],  
\end{align*} 
Following the proof of Proposition \ref{part size 8 nu - 1 mod 2ell}, the paths we can get with this partition of $D$  are
\begin{align*}
P_{20} &=(0,-20,9,46,10,45,11) &\quad   & Q_{20}=(0, -2,-1,11)\\
P_{21} &=(0,-21,7,49,8,48,9)   &\quad   & Q_{21}=(0, -4,-1,9)\\
P_{22} &=(0,-22,5,51,6,50,7)   &\quad   & Q_{22}=(0,-14,-1,7)\\
P_{23} &=(0,-23,3,53,4,52,5)   &\quad   & Q_{23}=(0,-16,-1,5) \\
P_{53} &=(0,-53,8,71,9,61,10)  &\quad   & Q_{53}=(0,-18,-1,10)\\
P_{54} &=(0,-54,6,73,7,72,8)   &\quad   & Q_{54}=(0,-25,-1,8)\\
P_{55} &=(0,-55,4,75,5,74,6)   &\quad   & Q_{55}=(0,-31,-1,6)\\
P_{56} &=(0,-56,2,77,3,76,4)   &\quad   & Q_{56}=(0,-33,-1,4)
\end{align*}
and joining them will give us the eight cycles making up the required difference family.
\end{ex}

\begin{lemma}\label{lemma:5subsets:2} 
Let $\ell=2\lambda + 3 \geq 5$ be odd, let $m\equiv \ell+1\pmod{2\ell}$ and set $s=4(m-1)/\ell$.
Then for any integer $\nu \geq 1$, there exists a partition of $[1, 4m\nu]\setminus ([1,4\nu]\cdot m)$ into nine subsets $X$ and
$A_i, A_i^*, B_i, Y_i$, for $i\in\{0,1\}$, which satisfy the following properties:
\begin{enumerate}
  \item $|X|= (\ell-5)\nu s$, $|A_i|=|A_i^*|=|B_i|=\frac{\nu s}{2}$,   $|Y_i| = \nu s$;
  \item there is a bijection $a\in A_i\mapsto a^* \in A_i^*$ such that 
  $B_i -\lambda +i -1 = \{a^*-a \mid a\in A_i\};$
  \item $X$ and $Y_1$ can be partitioned into pairs of consecutive integers;
  \item $Y_0$ can be partitioned into pairs at distance 2;
  \item $1\not\in A_0\cup A_1$, $B_i-\lambda +i -1\subset[2, 2m\nu-1]$, and $Y_0\ \cup\ Y_1\subset [1, 2m\nu-1]$   when  $\ell\geq 7$.    
\end{enumerate}
\end{lemma}
\begin{proof}
 Let $\epsilon=0$ or $1$ according to whether $\lambda$ is even or odd,
 and set $q=s\nu $.
 
 For every $h\in\{0,1,2\}$, set $A'_{h} = I_h + J_h\cdot m$, where the intervals $I_h, J_h$ and the integer  $\tau_{h}$ are defined as follows:
 \[
   \begin{array}{c|c|c|c}
     \hline \rule{0pt}{1\normalbaselineskip}
         &   I_h         & J_h               & \tau_h  \\[0.5ex] 
         \hline \hline \rule{0pt}{1\normalbaselineskip}
     h=0 & [-s/2, -1]    & [2\nu+1, 3\nu]    & (\nu-1)m   +s/2 +1 \\ [0.5ex] 
         \hline\rule{0pt}{1\normalbaselineskip}          
     h=1 &   [1, s/2-1]  & [2\nu, 3\nu-1]    & {\nu} m +s/2  \\[0.5ex]  
         \hline \rule{0pt}{1\normalbaselineskip}      
     h=2 &   \{-1\}      & [1,\nu]        & {\nu} m    \\[0.5ex]  \hline 
   \end{array}
 \]
 Also, let $B_0, B_1$, and $Y_0$ be the sets defined below:
 \begin{enumerate}
  \item[] $B_0 = \big(\big[3, s+1\big]_o   + \big[0, 2\nu-2\big]_e   \cdot m \big) + \lambda$
  \item[] $B_1 = \big(\big[0, s-2\big]_e   + \big[1, 2\nu-1\big]_o   \cdot m \big) + \lambda$,
  \item[] $Y_0 = Y_{0,0}\ \cup\ Y_{0,1}$, where $Y_{0,i} = B_i + (-1)^{i+1}(2\epsilon-1)$
 for $i\in\{0,1\}$. 
 \end{enumerate}
 It is not difficult to check that the sets $A'_h$ ($h \in \{0,1,2\}$), $A'_{k} + \tau_{k}$ ($k \in \{0,1,2\}$), $B_i$ ($i \in \{0,1\}$) and $Y_0$ are pairwise disjoint. 
 We denote by $W'$ their union, and note that
 $|A'_0|={\nu} s/2$, $|A'_1|=\nu (s/2-1)$, $|A'_2| = \nu$,  $|B_0|=|B_1|={\nu} s/2$, and $|Y_0|={\nu} s$;
 hence $|W'|=4{\nu} s$.
 
 Now set $W=D\setminus W'$ and note that $|W| = (\ell-4)s\nu$.
 Also, for every $j\in[0, 4\nu-1]$, it is not difficult to check that 
 $([1,m-1]+jm)\ \cap\ W$ is the disjoint 
 union of intervals of even size. Therefore, $W$ can be seen as the disjoint union of two subsets 
 $X$ and $Y_1$ each of which can be partitioned into pairs of consecutive integers,
 with $|X| = (\ell-5)\nu s$ and $|Y_1| = s\nu$.
 
 By construction, $Y_0\subset [1, 2m\nu -1]$ and it can be partitioned into pairs at distance $2$. Also, 
 since $\big|W' \cap [1, 2m\nu-1]\big| = 2\nu(s+1)$, we have that 
\begin{align*}
  \big|W \cap [1, 2m\nu-1]\big| &= 2\nu(m-1) - 2\nu(s+1) = 2\nu\big((\ell-4)s/4-1\big) \\
  &= s\nu \big((\ell-4)/2 - 2/s\big) \geq s\nu,
\end{align*}
when $\ell\geq 7$, in which case we can assume that $Y_1\subset[1, 2m\nu-1]$.

Finally, by Lemma \ref{lem:AA*}, there is a bijection
 $a\in A'_h \mapsto a^* \in A'_h+\tau_h$  such that
 \[
   \{a^* - a\mid a\in A'_h\} = 
   \begin{cases}
     \big[2, s\big]_e   + \big[0, 2\nu-2\big]_e   \cdot m  & \text{if $h=0$}, \\
     \big[2, s-2\big]_e + \big[1, 2\nu-1\big]_o   \cdot m  & \text{if $h=1$}, \\
     [1,2\nu-1]_o \cdot m                                  & \text{if $h=2$}.      
   \end{cases}
 \]
 Setting $A_0 = A'_0$, $A_0^* = A'_0+{\tau_0}$,  $A_1 = A'_1\ \cup\ A'_2$, and
 $A_1^* = (A'_1 +\tau_1)\ \cup\ (A'_2 + \tau_2)$, one can easily check that Condition 2 is satisfied, 
 and this completes the proof.
\end{proof}

\begin{prop} \label{part size 8 nu - ell+1 mod 2ell}
Let $\ell \geq 5$ be odd, and let $m\equiv \ell+1\pmod{2\ell}$. 
Then there is an $(8m\nu, 8\nu, \sC_{\ell})$-DF for every $\nu\geq 1$.
\end{prop}
\begin{proof}
We first consider the case $\ell=m-1\in\{5,7\}$. 
For every $i\in[1, 2\nu]$ and $j\in\{0,1\}$, set $x_i=2\nu+i$, $y_i=2i-1$, and
let $C^\ell_{i,j}$ be the following $\ell$-cycle:
\begin{align*}
C^5_{i,j} &= 
  \begin{cases}
    (0, 6x_i-1, 6y_i, -4, 6y_i-2) & \text{if $j=0$}, \\
    (0, 6x_i-3, 6y_i,  5, 6y_i+1) & \text{if $j=1$},
  \end{cases} \\
C^7_{i,j} &= 
  \begin{cases}
    (0, 8x_i-7, 8y_i,       3, 8y_i-1,       4, 8y_i-2)            & \text{if $j=0$}, \\
    (0, 8x_i-1, 8y_i, 16y_i+6, 8y_i+1, 16y_i+5, 8y_i+2) & \text{if $j=1$}.
  \end{cases} 
\end{align*}
Letting 
$\mathcal{C}^\ell=\{C^\ell_{i,j}\mid i\in[1, 2\nu], j\in\{0,1\}\}$, since
\begin{align*}
\Delta C^5_{i,0} &= \pm\{6x_i-1, 6(x_i-y_i)-1, 6y_i + 4,  6y_i + 2, 6y_i-2)\},\\
\Delta C^5_{i,1} &= \pm\{6x_i-3, 6(x_i-y_i)-3, 6y_i - 5,  6y_i - 4, 6y_i+1)\},\\
\Delta C^7_{i,0} &= \pm\{8x_i-7, 8(x_i-y_i)-7, 8y_i -3, 8y_i -4, 8y_i -5, 8y_i -6, 8y_i-2)\},\\
\Delta C^7_{i,1} &= \pm\{8x_i-1, 8(x_i-y_i)-1, 8y_i +6, 8y_i +5, 8y_i +4, 8y_i +3, 8y_i+2)\},
\end{align*}
and considering that $\{x_i-y_i \mid i\in[1, 2\nu]\} = \{2\nu-i+1\mid i\in[1, 2\nu]\}$, 
it follows that 
$\Delta \mathcal{C}^\ell = \pm [1, 4m\nu]\setminus (\pm[1,4\nu]\cdot m) = \Z_{8m\nu}\setminus(m\cdot\Z_{8m\nu})$, 
hence $\mathcal{C}^\ell$ is a set of base cycles for a cyclic $\ell$-cycle decomposition of $K_m[8\nu]$.

We now assume that $(\ell,m)\not\in\{(5,6),(7,8)\}$. 
Set $\lambda=(\ell-3)/2$ and let $\epsilon\in\{0,1\}$, with $\epsilon\equiv \lambda \pmod{2}$. Also,
set $s=4(m-1)/\ell$ and $q={\nu} s$, and note that the number of base cycles in the difference family is $q$.  
By Lemma \ref{lemma:5subsets:2} there is a partition of
$D=[1, 4m\nu]\setminus ([1,4\nu]\cdot m)$ into nine subsets, $X$ and
$A_i, A_i^*, B_i, Y_i$, for $i\in\{0,1\}$, which satisfy the following properties:
\begin{enumerate}
  \item $|X|= (\ell-5)\nu s$, $|A_i|=|A_i^*|=|B_i|=\frac{\nu s}{2}$,   $|Y_i| = \nu s$;
  \item there is a bijection $a\in A_i\mapsto a^* \in A_i^*$ such that 
  $B_i-\lambda + i -1 = \{a^*-a \mid a\in A_i\}$;
  \item $X$ and $Y_1$ can be partitioned into pairs of consecutive integers;
  \item $Y_0$ can be partitioned into pairs at distance 2;
  \item $1 \notin A_0 \cup A_1$, $B_i-\lambda + i -1\subset [2, 2m\nu-1]$, and $Y_0\ \cup\ Y_1\subset [1, 2m\nu-1]$   when  $\ell\geq 7$.    
\end{enumerate}
In particular, $X$ can be seen as the disjoint union of $q$ sets $X_a$ of size $\ell-5$, 
indexed over the elements of $A_0\ \cup\ A_1$, 
each of which can be partitioned into pairs of consecutive integers. Also, we can write
$Y_0 = \{y_a, y_a - 2 \mid a\in A_0\}$ and  $Y_1 = \{y_a, y_a -1 \mid a\in A_1\}$. 

By applying Lemma \ref{lemma:odd1} 
with $d=a \in A_0\ \cup\ A_1$ and $\mathbb{X}=X_a$, we construct a path $P_a$ of length $2\lambda = \ell-3$
such that
\begin{eqnarray}
  & \label{2Pa:ends}&\text{the ends of $P_a$ are $0$ and $p_a= a^*- a + \lambda-1$,}\\
  & \label{2Pa:vertices}&\text{$V(P_a)\subseteq \{0,-a\}\ \cup\ [a^*-a, a^*-a + \max{X_a}]$}, \\
  & \label{2Pa:delta}&\text{$\Delta P_a = \pm \{a, a^*\}\ \cup\ \pm X_a$},
\end{eqnarray}
where $\max X_a=0$ when $\ell=5$.
For $i\in\{0,1\}$ and $a\in A_i$, let $C_a$ be the closed trail obtained by joining $P_a$ and 
the $3$-path $Q_a = 0,\; -y_a,\; i-2,\; p_a$,
and considering its vertices as elements of $\mathbb{Z}_{8m\nu}$. 
We claim that $\mathcal{C} = \{C_a\mid a\in A\}$, is the desired difference family.

We first show that $\Delta \mathcal{C} = \pm D$. 
Recalling \eqref{2Pa:delta}, and  considering that 
\[B_i =\{a^* -a +\lambda -i + 1 \mid a\in A_i\} = \{p_a+2-i\mid a\in A_i\}\]
and $\Delta Q_a = \pm \{y_{a}, y_{a}-2+i, p_a+2-i\}$, for every $i\in\{0,1\}$ and $a\in A_i$,
then
\begin{align*}
  \Delta \mathcal{C} &= \bigcup_{i\in \{0,1\}}\bigcup_{a\in A_i} \Delta C_a =
  \bigcup_{i\in \{0,1\}}\bigcup_{a\in A_i} \big(\Delta P_a\ \cup\ \Delta Q_a\big)  \\
   & =\pm \bigcup_{i\in \{0,1\}}\bigcup_{a\in A_i} \big(X_a \ \cup\ \{a, a^*, y_{a}, y_{a}-2+i, p_a+2-i\}\big)\\
   & =\pm (D\setminus (B_0\ \cup\ B_1)) \ \cup\ \pm\big\{p_a+2-i\mid i\in\{0,1\}, a\in A_i\big\} = \pm D.
\end{align*}

We finish by showing that $C_a$ is a cycle. 
Recalling \eqref{2Pa:vertices}, if $\ell=5$, then  
$V(P_a) = \{0, a, p_a=a^*-a\}$. Considering that
$A_0, A_1, Y_0, Y_1\subset [1, 4m\nu -1]$, it follows that 
$a\not\in\{-y_a, -1\}$, hence $C_a$ is a cycle. 
Again by \eqref{2Pa:vertices}, if $\ell\geq7$, then
$V(P_a) \subseteq \{0,-a\}\ \cup\ [a^*-a, a^*-a + \max X_a]$. 
By conditions 2 and 4, we have that
$a^*-a\in B - \lambda +i -1 \subset [2, 2m\nu-1]$ for every $a\in A_i$. Since
$\max X_a<4m\nu$, then
$V(P_a) \subseteq \{0,-a\} \ \cup\ [2, 6m\nu-1]$ for every $a\in A_0 \cup A_1$. 
Recalling that $1\not\in A_0 \cup A_1$ and $Y_0\ \cup\ Y_1\subset [1, 2m\nu-1]$, 
and that $A, Y_0$ and $Y_1$ are pairwise disjoint, it follows that $\{-1, -y_a\}\not\in V(P_a)$, therefore
$P_a$ and $Q_a$ only share their end-vertices, 
hence $C_a$ is a cycle, for every $a\in A_0\ \cup A_1$,
and this completes the proof.
\end{proof}

\begin{ex}
Let $\ell=9$, $m=10$ and $n=8$, so that $\nu=1, \lambda=3, \epsilon=1, s=4$, 
and the size of our $(80, 8, \sC_{9})$-DF will be $q=4$.
Following the proof of Lemma \ref{lemma:5subsets:2},  we have that
\[
A_0=\{28,29\}, A_1=\{9,21\}, A_0^*=\{31,32\}, A_1^*=\{19,33\}, 
\]
\[
B_0=\{6,8\}, B_1=\{13,15\}, Y_{0,0}= B_0-1=\{5,7\}, Y_{0,1}= B_1+1=\{14,16\}.
\]
Our index set is $A=\{9,21,28,29\}$, and
 $Y_0 = Y_{0,0}\ \cup \ Y_{0,1} = \{5,7, 14,16\}.$
We can now choose, for instance, $Y_1=[1,4]$ and $X_{28}=\{11,12,17,18\}, X_{29}=[22,25], X_{9}=\{26,27,34,35\}, 
X_{21}=[36,39]$. 

Following the proof of Proposition \ref{part size 8 nu - ell+1 mod 2ell}, the paths we can get with this partition of 
$D=[1, 80]\setminus \big([1,8]\cdot 10\big)$  are
\begin{align*}
P_{28} &=(0,-28,4,22,5,17,6)&\quad & Q_{28}=(0,-7,-2,6)\\
P_{29} &=(0,-29,2,27,3,26,4)&\quad & Q_{29}=(0,-16,-2,4)\\
P_{9}  &=(0,-9,10,45,11,38,12)&\quad & Q_{9}=(0,-2,-1,12)\\
P_{21} &=(0,-21,12,51,13,50,14)&\quad & Q_{21}=(0,-4,-1,14),
\end{align*}
and joining them will give us the four cycles making up the required difference family.
\end{ex}

\section{The proof of Theorem \ref{ell divides n}}\label{odd2}

The aim of this section is to prove Theorem \ref{ell divides n}.
The case $n=2\ell$ is treated in Proposition \ref{part size 2l}, while the case $n\equiv 0 \pmod{4\ell}$ is dealt with in Proposition \ref{part size 0 mod 4ell - m odd} for $m$ odd, and in Proposition 
\ref{part size 0 mod 4ell - m even} for $m$ even.

These results will be proved using a strategy very similar to the one used in 
Section \ref{odd1}. Once more, we partition (in Lemmas \ref{lemma:dividesn:1}, \ref{lemma:dividesn:2}) the set $D$
of differences to be realized into various sets, 
namely $D=A\ \cup\ A^*\ \cup\ B\ \cup\ Y\ \cup\ W$, where 
$A$ is a set of size $q$, the cardinality of the DF,  that will serve as the set of indices for the cycles in the DF;  it will be paired up with a second $q$-set, the set $A^*$ chosen with the help of Lemma \ref{lem:AA*}. 
To build the cycle $C_a$, to each pair of elements $\{a, a^*\}$ we will associate a set $W_a\subset W$ of size 
$\ell-5$ that can be partitioned into pairs at distance $m$, and a pair of  integers $\{y_a,y_a'\}$ from 
$Y$, in such a way that $\Delta C_a = \pm\{a, a^*, y_a, y_a', \delta_a\}\ \cup \ \pm W$,
where $\delta_a$ is an integer depending on $a$. Since the pairs $\{a, a^*\}$, $\{y_a,y_a'\}$ are chosen to partition, between them, $A\cup A^*$ and $Y$, respectively, while the sets $W_a$ partition $W$, then
$\cup_{a\in A} \Delta C_a = \pm (D\setminus B)\ \cup\ \pm\{\delta_a\mid a\in A\}$. By showing
that $\{\delta_a\mid a\in A\}=B$, we prove that $\{C_a\mid a\in A\}$ is the desired difference family. 

We choose to also prove Proposition \ref{part size 2l}, the case $n=2\ell$, with this strategy to help familiarize the reader with the techniques we use later in Proposition \ref{part size 0 mod 4ell - m odd} and in Proposition \ref{part size 0 mod 4ell - m even}.  In this particular case, a simpler proof using Rosa sequences - a variation of Skolem sequences - is also possible, but adapting such an approach to the general case is not straightforward.

\begin{prop} \label{part size 2l}
  If $\ell=2\lambda+3 \geq 5$ and $m \equiv 0, 1 \pmod{4}$, 
  then there exists a $(2\ell{m}, 2\ell, \sC_\ell)$-DF.
\end{prop}
\begin{proof}
 Let $\epsilon=0$ or $1$ according to whether $\lambda$ is even or odd,
 and let $q=(m-1) $; the difference family will have size $q$.

We begin by considering the case that either $m>4$ or $(m,\epsilon)=(4,1)$.
We first partition the set $D = [1,m-1] + [0, \ell-1]m$ of differences to be realized into
 five subsets $A, A^*$, $B, Y$ and $W$, with $|A|=|A^*|=|B|=q$, $|Y|=2q$, $|W|=(\ell-5)q$.
 
We will set $A=A_{-1} \cup A_0\cup A_1$, and $A^*=A^*_{-1} \cup A^*_0\cup A^*_1$, where the sets $A_{i}$, $A_i^*$ for $i\in[-1,1]$ and $B$ are defined as follows.  If $m$ is odd, then 
 \[
A_0 = \varnothing \mbox{ and } A_0^* = \varnothing,
\]
and if $m$ is even, then
\[
A_0 = \{m-1 + (\lambda -2 +2\epsilon)m\} \mbox{ and } A_0^* = A_0 + (1-\epsilon)m+2.
\]
In any case,
\[
 \begin{array}{ll}
   A_1   = \left[1, \lfloor \frac{m-1}{2}\rfloor \right] + (\lambda -2 +2\epsilon)m,
&   
   A^*_1 = A_1 + \lfloor \frac{m-1}{2}\rfloor,\\
   \rule{0pt}{1.5\normalbaselineskip}
   A_{-1}   = \left[\lceil \frac{m+1}{2}\rceil, m-1 \right] + (\ell-2)m,
& 
   A^*_{-1} = A_{-1} + \lfloor \frac{m+1}{2}\rfloor,  \\
     \rule{0pt}{1.5\normalbaselineskip}
   \;\,B= [1,m-1]+(\lambda-1)m.
&   
 \end{array}
\] 

Either directly, or by applying Lemma \ref{lem:AA*}, it is easy to see that there is a bijection
$a\in A_i \mapsto a^* \in A^*_{i}$, for $i\in[-1,1]$, such that
 \begin{equation}\label{a*-a}
 \begin{aligned}
  \{a^* - a\mid a\in A_i\} &=    
     \begin{cases}
       \varnothing          &  \text{if $i=0$ and $m$ is odd}, \\  
       [1, m-2]_o           &  \text{if $i=1$ and $m$ is odd}, \\   
       [2, m-1]_e           &  \text{if $i=-1$ and $m$ is odd}, \\              
       (1-\epsilon)m +2     &  \text{if $i=0$ and $m$ is even}, \\  
       [1, m-3]_o           &  \text{if $i=1$ and $m$ is even}, \\        
       [2, m-2]_e           &  \text{if $i=-1$ and $m$ is even}.       
     \end{cases}
  \end{aligned}
  \end{equation}
  Therefore,
  \begin{equation}\label{A*-A}
  \begin{aligned}
    \{a^*- a +i\mid i=\pm 1, a\in A_i\} \cup \{a^*- a + (\epsilon{m}-3)\mid a\in A_0\} \\
    = [1,m-1]=  B-(\lambda-1)m.
  \end{aligned}
  \end{equation}

  Letting $D' = [1, m-1] + \{\lambda -2 + 2\epsilon, \lambda+\epsilon, \ell-2, \ell-1\}\cdot{m}$,
  $A=A_{-1} \cup A_0\cup A_1$, and $A^*=A^*_{-1} \cup A^*_0\cup A^*_1$, we notice that
  $A \cup A^* \subset D'$ and $B\cap D'= \varnothing.$
 Also, the set $Y = D' \setminus (A\ \cup A^*)$ has size $2m-2 = 2|A|$ and 
 it is the disjoint union of the following three intervals:
 \begin{equation}\label{Y123}
 \begin{aligned}
   Y_1 &= (\lambda+\epsilon)m +
   \begin{cases}
     [1, m-1] &\text{if $m$ is odd},\\   
     [2, m-1] &\text{if $m$ is even},
   \end{cases}\\
   Y_2 &= \left[1, \left\lceil\frac{m-1}{2}\right\rceil\right]   + (\ell-2)m,\;\;\;
   Y_3 = \left[\left\lfloor \frac{m+1}{2} \right\rfloor, m-1\right] + (\ell-1)m.
 \end{aligned}
 \end{equation}
 Recalling that $m\equiv 0,1 \pmod{4}$, it follows that $Y_2$ and $Y_3$ have even size.  Therefore, $Y$ can be partitioned into pairs of consecutive integers and a pair $\{y', y''\}\subset Y_1$ such that $\pm (y'-y'') = \pm (\epsilon m-3)$. 
Using the elements of
 $A$ to index such pairs, we can thus write
 \begin{equation}\label{Y}
   Y=\big\{y_a, y_a-1\mid a\in A\setminus A_0\big\}\ \cup\ \{y_a, y_a -(\epsilon{m} -3)\mid a\in A_0\}.
 \end{equation} 
 Finally, let $W = D\setminus(D' \cup B)$ and note that $W=[1,m-1]+(U_1 \cup U_2)m$ where
 \begin{equation}\label{W0}
    U_1 = [0, \lambda-3+\epsilon]\;\;\text{and\;\;} U_2 = [\lambda+1+\epsilon, \ell-3].
 \end{equation}
 Since  both $U_1$ and $U_2$ have even size, and $|U_1|+|U_2|=\ell-5$, 
 it follows that $W$ can  be partitioned into $m-1=q$ subsets  $\{W_a\mid a\in A\}$ each of size $\ell-5$ such that 
 \begin{align}
  \label{W1}
   & W_a =\{w_{a, t}, w_{a, t}-m \mid t\in[1, \lambda-1]\},\\ 
   \label{W2}
   & \text{$w_{a, t} \geq w_{a, t+1} + 2m$ and $w_{a, t}\not\equiv a\; ({\rm mod}\; m)$
           for every $t\in[1, \lambda-2]$.}
 \end{align} 
 
 We use the partition $\{A, A^*,B,Y, W\}$ of $D$ to construct the desired
 difference family.
 Let $\mathcal{F}=\{C_a\mid a\in A\}$, with 
$C_a = (c_{a,0}, c_{a,1}, \ldots, c_{a,\ell-1})$, 
be a set of $q$ closed trails of length $\ell$ defined as follows:
\begin{align*}
(c_{a,0}, c_{a,1}, c_{a,2}) &= (0,a^*, a^*-a),\\
(c_{a,\ell-2}, c_{a,\ell-1}) &= 
\begin{cases}
  (-1, -y_a) & \text{if $a\in A_{1}$},\\
  (1, -y_a + 1) & \text{if $a\in A_{-1}$},\\
  (3-\epsilon{m}, -y_a) & \text{if $a\in A_{0}$},\\  
\end{cases}\\
c_{a, u} &= a^*-a +
\begin{cases}
  w_{a, \frac{u-1}{2}} + \frac{u-3}{2}m     & \text{if $u\in[3, \ell-4]$ is odd}, \\
  \frac{u-2}{2}m       & \text{if $u\in[4, \ell-3]$ is even},  
\end{cases}
\end{align*}
We claim that $\mathcal{F}$ is a $(2\ell{m}, 2\ell, \sC_\ell)$-DF, that is, 
$\Delta \mathcal{F} = \mathbb{Z}_{2\ell{m}}\setminus(m\mathbb{Z}_{2\ell{m}})$ and 
the vertices of each $C_a$ are pairwise distinct. 
For   $i\in[-1,1]$ and $a\in A_{i}$, we have 
\[
  \Delta C_a = \pm W_a \cup
  \begin{cases} 
    \pm\left\{a, a^*, y_a, y_a-1, c_{a,\ell-3}+i\right\} & \text{if $i=\pm1$,}\\
    \pm\left\{a, a^*, y_a, y_a-(\epsilon{m}-3), c_{a,\ell-3}+(\epsilon{m}-3)\right\} & \text{if $i=0$.}    
  \end{cases}
\] 
Since
$c_{a,\ell-3} = a^*- a + (\lambda-1) m$, by condition \eqref{A*-A}
it follows that 
\[
  \{c_{a,\ell-3}+i \mid i=\pm1, a\in A_{i}\} \ \cup \{c_{a,\ell-3}+(\epsilon{m}-3)\mid a\in A_0\}= B.
  \] 
Recalling also conditions \eqref{Y} and \eqref{W1}, we have that
$\Delta \mathcal{C} = \pm D = \mathbb{Z}_{mn}\setminus(m\mathbb{Z}_{mn})$.
Finally, we have to show that each $C_a$ does not have repeated vertices. 
By \eqref{a*-a}, $a^*-a\in[1,m+2]$, and
by conditions \eqref{W0} and \eqref{W2}
we have that $(\ell-3)m < w_{a,1} < (\ell-2)m$ and $w_{a,t}\geq w_{a,t+1} + 2m$ 
for every $a\in A$ and $t\in [1, \lambda-2]$. Hence
\begin{align*}
  \ell{m} &> a^*-a + w_{a,1} = c_{a,3} >  c_{a,5} > \ldots > c_{a,\ell-6} \\
    &> c_{a,\ell-4} =  a^*-a + w_{a,\lambda-1} + (\lambda-2)m \\
  &> a^*-a  +  (\lambda-1)m = c_{a, \ell-3} > c_{a, \ell-5} > \ldots > c_{a,6}  \\
  &> c_{a,4} = a^*-a + 2m > a^*-a = c_{a,2}\geq 1.
\end{align*}
Also, by \eqref{Y123} $c_{a, \ell-1}\in[-\ell{m}, -\lambda{m}]$.
Recalling that $c_{a, \ell-2}\in\{3-m, -1, 1,3\}$,
if $c_{a,2}=a^*-a\in\{1,3\}$, then $a\in A_1$ by \eqref{a*-a}, hence $c_{a, \ell-2}=-1$. 
Therefore, $c_{a,0}=0$ and $c_{a,2}, c_{a,3}, \ldots, c_{a, \ell-1}$ are pairwise distinct.
By \eqref{W2} and considering that $a^*\not\equiv 0\; ({\rm mod}\; m)$ and $a^*\in[2, \ell{m}-1]$, 
we have that $a^*=c_{a,1}\neq c_{a,u}$ for every $u\in[0, \ell-1]$. 
It follows that 
each $C_a$ is a cycle, and this completes the proof provided $(m,\epsilon) \neq (4,0)$.

The case $m=4$ and $\epsilon=0$ is similar, except that in $D$ and $D'$ we replace the difference $\ell m - 1 = 4\ell-1$ with $\ell m + 1 = 4\ell+1$, and partition the set $Y=D' \setminus (A \cup A^*)$ into intervals $Y_1= [2,3] + {4}\lambda$ and $Y_2=[1,2] + 4(\ell-2)$ 
as before together with the set
\[
Y_3 = \left(\left[ \left\lfloor \frac{m+1}{2} \right\rfloor, m-2 \right] + (\ell-1) m \right) \cup \{ \ell m + 1\} = \{4\ell-2, 4\ell+1\}.
\]
Note that $Y_1$ and $Y_2$ each consist of consecutive integers, so that $Y$ is partitioned into $q-1=2$ pairs of consecutive integers and one pair $\{y',y''\}$ satisfying $\pm (y'-y'') = \pm 3 = \pm (\epsilon m-3)$.  The remainder of the proof proceeds as before.
\end{proof}  
    
\begin{ex}
\begin{enumerate}
\item
Let $\ell=9$ and $m=5$. We have $A_1=\{16,17\}$, $A_{-1}=\{38,39\}$, $A_1^*=\{18,19\}$, $A_{-1}^*=\{41,42\}$, $B=\{11,12,13,14\}$, $Y=[21,24] \cup \{36,37\} \cup \{43,44\}$, so that $W=[1,4]\cup[6,9]\cup[26,29]\cup[31,34]$.
The $(2\ell{m}, 2\ell, \sC_\ell)$-DF consists of the following four cycles.
\begin{align*}
C_{16}&=(0,19,3,37,8,17,13,-1,-24)\\
C_{17}&=(0,18,1,34,6,14,11,-1,-22)\\
C_{38}&=(0,42,4,36,9,16,14,1,-43)\\
C_{39}&=(0,41,2,33,7,13,12,1,-36)
\end{align*}

\item Let $\ell=7$ and $m=8$. We have  $A_1=\{1,2,3\}$, $A_0=\{7\},$ $A_{-1}=\{45,46,47\}$, $A_1^*=\{4,5,6\}$,$A_0^*=\{17\},$ $A_{-1}^*=\{49,50,51\}$, $B=[9,15],$ $Y=[18,23]\cup [41,44]\cup [52,55]$,  so that $W=[25,31]\cup[33,39]$.
The $(2\ell{m}, 2\ell, \sC_\ell)$-DF consists of the following seven cycles.
\begin{align*}
C_{1}  &=(0, 6, 5, 39, 13, -1, -20)&\quad  & C_{2} =(0,  5,  3, 38, 11, -1, -23)\\
C_{3}  &=(0, 4, 1, 37,  9, -1, -42)&\quad  & C_{7} =(0, 17, 10, 47, 18,  3, -18)\\
C_{45} &=(0,51, 6, 44, 14,  1, -43)&\quad & C_{46} =(0, 50,  4, 43, 12,  1, -52)\\
C_{47} &=(0,49, 2, 35, 10,  1, -54).&\quad & 
\end{align*}

\item
Let $\ell=7$, $m=4$ and $n=14$.  We have $A_0=\{3\}$, $A_0^*=\{9\}$, $A_1=\{1\}$, $A_1^*=\{2\}$, $A_{-1}=\{23\}$, $A_{-1}^*=\{25\}$, $B=[5,7]$, $Y=\{10,11\} \cup \{21,22\} \cup \{26,29\}$ and $W=[13,15] \cup [17,19]$.  The $(2\ell{m}, 2\ell, \sC_\ell)$-DF consists of the following three cycles:
\begin{align*}
C_{1}  &= (0,  2, 1, 20,  5, -1, -22) \\
C_{3}  &= (0,  9, 6, 23, 10,  3, -26) \\
C_{23} &= (0, 25, 2, 20,  6,  1, -10)
\end{align*}
\end{enumerate}
\end{ex}

We now deal with the case where $n\equiv 0 \pmod{4\ell}$.
The partitions of $D$ (i.e., the set of differences to be realized) outlined in the beginning of this section are given in Lemmas \ref{lemma:dividesn:1} and \ref{lemma:dividesn:2}.

\begin{lemma}\label{lemma:dividesn:1} 
Let $\ell=2\lambda+3 \geq 5$ be odd, let $m\ge3$ be odd and 
$n=4\ell\nu$ with $\nu\geq1$. 
Then there exists a partition of $[1, mn/2]\setminus ([1,n/2]\cdot m)$
into seven subsets,
$A_{i}, A^*_i$, for $i=\pm1$,  and $B, Y$, $W$, satisfying the following properties:
\begin{enumerate}
  \item\label{lemma:dividesn:1:1}   
      $|A_{i}| =|A^*_i|=(m-1)\nu$ for $i=\pm 1$, $2|B| = |Y| = 4(m-1)\nu$, and
      $|W|= 2(\ell-5)(m-1)\nu$;
  \item \label{lemma:dividesn:1:2} 
      there is a bijection $a\in A_i\mapsto a^* \in A_i^*$, for $i=\pm1$, such that 
      \[B-(\lambda-1)m= \big\{a^*-a+i\mid i=\pm 1, a\in A_{i}\big\};\]
  \item \label{lemma:dividesn:1:2b} $a^*-a\geq 2$ for every $a\in A_{-1}\cup A_{1}$;
  \item \label{lemma:dividesn:1:3} 
      $Y$ can be partitioned into pairs of consecutive integers;
  \item \label{lemma:dividesn:1:4} 
      $W$ can  be partitioned into $2(m-1)\nu$ sets  $\{W_a\mid a\in A_{-1}\cup A_{1}\}$ 
      each of size $\ell-5$ such that  
      \begin{enumerate}
      \item  $W_a =\{w_{a, t}, w_{a, t}-m \mid t\in[1, \lambda-1]\}$, and
      \item  $a> w_{a, t} \geq w_{a, t+1} + 2m$ for every $t\in[1, \lambda-2]$.
      \end{enumerate}
\end{enumerate}
\end{lemma}
\begin{proof}
 Let $\epsilon=0$ or $1$ according to whether $\lambda$ is even or odd,
 and let $q=2(m-1)\nu $. We start by defining the intervals $I_h, J_h$ and the integer $\tau_{h}$, for $h\in[0,4]$, 
 as follows: 
 \[
   \begin{array}{c|c|c|c}
     \hline \rule{0pt}{1\normalbaselineskip}
         &   I_h       & J_h               & \tau_h  \\[0.5ex] 
             \hline \hline \rule{0pt}{1\normalbaselineskip}
     h=0 & \multirow{2}{*}{$\left[1, \frac{m-1}{2}\right]$} 
         & \left[(2\ell-4)\nu, (2\ell-3)\nu -1\right]  
         & \nu m +\frac{m-3}{2} \\[0.5ex]  
         \cline{1-1}  \cline{3-4} \rule{0pt}{1\normalbaselineskip}
     h=1 &  
         & \left[(2\ell-2)\nu, (2\ell-1)\nu -1\right]    
         & {\nu}m  +\frac{m-1}{2} \\ [0.5ex] 
         \hline \rule{0pt}{1\normalbaselineskip}
     h=2 & \left[\frac{m+1}{2}, m-2\right]         
         & \multirow{2}{*}{$\left[(2\ell-4)\nu,(2\ell-3)\nu -1\right]$}    
         & {\nu}m +\frac{1-m}{2}  \\[0.5ex]  
         \cline{1-2}  \cline{4-4} \rule{0pt}{1\normalbaselineskip}
     h=3 & \{m-1\}     
         &
         & {\nu}m \\[0.5ex]  \hline   
         \rule{0pt}{1\normalbaselineskip}
     h=4 & \left[\frac{m+1}{2}, m-1\right]  
         & \left[(2\ell-2)\nu, (2\ell-1)\nu -1\right]      
         & {\nu}m +\frac{1-m}{2}    \\[0.5ex]  
         \hline
   \end{array}
 \]
 For every $h\in[0,4]$, set $A'_{h} = I_h + J_h\cdot m$, 
 $A'_{h+5} = A'_{h} + \tau_h$ 
 Also, set $B=[1,m-1] + [\lambda-1,\lambda-2+2\nu]\cdot m$.
 Furthermore, by Lemma \ref{lem:AA*}, there is a bijection
 $a\in A'_h \mapsto a^* \in A'_{h+5}$, for $h\in[0,4]$, such that
 \begin{align*}
  \{a^* - a\mid a\in A'_h\} &=    
     \begin{cases}
       \big[0, m-2]_e + \big[0, 2\nu-1\big]_o   \cdot m      &  \text{if $h=0$}, \\  
       \big[0, m-2\big]_o + \big[0, 2\nu-1\big]_o \cdot m    &  \text{if $h=1$}, \\  
       \big[2, m-1\big]_o +  \big[0, 2\nu-1\big]_e \cdot m   &  \text{if $h=2$}, \\  
            \{m\} + \big[0, 2\nu-1\big]_e \cdot m            &  \text{if $h=3$}, \\  
       \big[2,m]_e +\big[0, 2\nu-1\big]_e \cdot m            &  \text{if $h=4$}.        
     \end{cases}
  \end{align*}
 Considering that the sets $A'_h$ and $B$ are pairwise disjoint, it is not difficult to check that $B$ and
 the sets $A_{1} = A'_0\cup A'_1$, $A^*_{1} = A'_5\cup A'_6$,
 $A_{-1}= A'_2\cup A'_3\cup A'_4$, and $A^*_{-1}= A'_7\cup A'_8\cup A'_9$ satisfy conditions 
 \ref{lemma:dividesn:1:1}--\ref{lemma:dividesn:1:2b}.

Now, denoting by $\ol{D}$ the set of all elements of $[1, mn/2]\setminus ([1,n/2]\cdot m)$ 
not lying in any of the sets $A_{i}, A^*_i$, for $i=\pm 1$, or $B$, we have that $\ol{D}$ can be partitioned into 
the following two subsets:
\begin{align*}
  Y &= [1,m-1] + \big([\lambda - 1 + 2\nu, \lambda - 3 + 6\nu]\ \cup\ \{\lambda -2 + 6\nu\epsilon\}\big)\cdot m, \\
  W &= [1,m-1] + (U_1\ \cup\ U_2)m,
\end{align*}
where $U_1 = [0, \lambda -3 +\epsilon]$ and $U_2 = [\lambda-2+6 \nu+ \epsilon, (2\ell-4)\nu-1]$.

Note that $Y$ has size $4(m-1)\nu = 2|A_{-1}\ \cup\ A_1|$ and it is the disjoint union of $4\nu$ intervals of size $m-1\equiv 0\pmod{2}$; hence $Y$ can be partitioned into 
$|A_{-1}\ \cup\ A_1|$  pairs of consecutive integers, therefore
condition \ref{lemma:dividesn:1:3} holds.

Finally, since $U_1$ and $U_2$ have even size, and $|U_1\ \cup\ U_2| = 2(\ell-5)\nu$, then
$W$ has size $2(\ell-5)(m-1)\nu = (\ell-5)|A_{-1}\ \cup\ A_1|$ and there exists a partition
$\{W_a\mid a\in A_{-1}\ \cup\ A_1\}$ of $W$ satisfying condition \ref{lemma:dividesn:1:4}, and this completes the proof. 
\end{proof}

\begin{prop} \label{part size 0 mod 4ell - m odd}
Let $\ell \geq 5$ be odd, and let $n\equiv 0 \pmod{4\ell}$. 
There exists a $(nm,n,{\sC}_\ell)$-DF for every odd $m\geq 3$.
\end{prop}
\begin{proof} Set $D=[1, mn/2]\setminus ([1,n/2]\cdot m)$, let $n=4\ell\nu$ with $\nu\geq1$, and set 
$q=2(m-1)\nu$, noting that $q$ is the size of the difference family to be constructed. Also, set $\lambda=(\ell-3)/2$ and
let $\epsilon=0$ or $1$ according to whether $\lambda$ is even or odd.
By Lemma \ref{lemma:dividesn:1}, there is a partition of
$D=[1, mn/2]\setminus ([1,n/2]\cdot m)$ into seven subsets,
$A_i$, $A_i^*$, for $i=\pm 1$, and $B$, $Y$, $W$ which satisfy the conditions
\ref{lemma:dividesn:1:1}--\ref{lemma:dividesn:1:4} of Lemma \ref{lemma:dividesn:1}.

By condition \ref{lemma:dividesn:1:3}, we can write
$Y=\{y_a, y_a-1\mid a\in A_1\cup A_{-1}\}$.
Now, let $\mathcal{F}=\{C_a\mid a\in A_1\cup A_{-1}\}$, with 
$C_a = (c_{a,0}, c_{a,1}, \ldots, c_{a,\ell-1})$, 
be a set of $q$ closed trails of length $\ell$ defined as follows:
\begin{align*}
(c_{a,0}, c_{a,1}, c_{a,2}) &= (0,a^*, a^*-a),\\
(c_{a,\ell-2}, c_{a,\ell-1}) &= 
\begin{cases}
  (-1, -y_a) & \text{if $a\in A_{1}$},\\
  (1, -y_a + 1) & \text{if $a\in A_{-1}$},
\end{cases}\\
c_{a, u} &= a^*-a +
\begin{cases}
  w_{a, \frac{u-1}{2}} + \frac{u-3}{2}m     & \text{if $u\in[3, \ell-4]$ is odd}, \\
  \frac{u-2}{2}m       & \text{if $u\in[4, \ell-3]$ is even},  
\end{cases}
\end{align*}
We claim that $\mathcal{F}$ is the desired difference family, that is, 
$\Delta \mathcal{F} = \mathbb{Z}_{mn}\setminus(m\mathbb{Z}_{mn})$ and the vertices of each $C_a$ are pairwise distinct.

For   $i=\pm1$ and $a\in A_{i}$, we have that
\[\Delta C_a = \pm\left\{a, a^*, y_a, y_a-1, c_{a,\ell-3}+i\right\}\ \cup\ W_a.
\] 
Since
$c_{a,\ell-3} = a^*- a + (\lambda-1) m$, by condition \ref{lemma:dividesn:1:2}
it follows that $\{c_{a,\ell-3}+i, \mid i=\pm1, a\in A_{i}\} = B$, therefore 
$\Delta \mathcal{C} = \pm D = \mathbb{Z}_{mn}\setminus(m\mathbb{Z}_{mn})$.

Finally, considering that, 
by conditions \ref{lemma:dividesn:1:2b} and \ref{lemma:dividesn:1:4} 
of Lemma \ref{lemma:dividesn:1},
$m < w_{a,1} < a$, $w_{a,t}\geq w_{a,t+1} + 2m$, and $a^*-a\geq 2$, for every $a\in A_1\cup A_{-1}$ and $t\in [1, \lambda-2]$,
it follows that
\begin{align*}
  mn/2> a^* =c_{a,1} &> a^*-a + w_{a,1} = c_{a,3} >  c_{a,5} > \ldots > c_{a,\ell-6} \\
    &> c_{a,\ell-4} =  a^*-a + w_{a,\lambda-1} + (\lambda-2)m \\
  &> a^*-a  +  (\lambda-1)m = c_{a, \ell-3} > c_{a, \ell-5} > \ldots > c_{a,6}  \\
  &> c_{a,4} = a^*-a + 2m > a^*-a = c_{a,2}\geq 2
\end{align*}
and this guarantees that each $C_a$ is a cycle.
\end{proof}

\begin{ex}
Let $\ell=7$, $m=5$, and $n=28$, so that $\nu=1, \lambda=2, q=8$.
We have 
\begin{align*}
A'_0 & = &\{51,52\},    \quad  A'_5 & =&\{57,58\},    \quad A'_1 & = &\{61,62\},   \quad A'_6 & = &\{68,69\}, \\
A'_2 & = &\{53\},         \quad  A'_7 & =&\{56\},         \quad A'_3 & = &\{54\},        \quad A'_8 & = &\{59\}, \\
&&                                        A'_4 & = &\{63,64\},    \quad A'_9 & =&\{66,67\},    \quad                                   \quad & 
\end{align*}
so that $A_1=\{51,52,61,62\},A_1^*=\{57,58,68,69\},A_{-1}=\{53,54,63,64\},$ and $A_{-1}^*=\{56,59,66,67\}.$

Also, $B=[6,9]\cup[11,14]$ and $Y=[1,4]\cup[16,19]\cup[21,24]\cup[26,29]$, so that $W=[31,49]\setminus \{35,40,45\}$. For the sets $W_a$, $a\in A_1\cup A_{-1}$ we can choose for instance
\begin{align*}
W_{51}&=&\{36,31\},\quad y_{51}&=&2   \quad W_{52}&=&\{37,32\},\quad y_{52}&=&4   \quad \quad\\
W_{61}&=&\{38,33\},\quad y_{61}&=&17 \quad W_{62}&=&\{39,34\},\quad y_{62}&=&19 \quad \quad\\  
W_{53}&=&\{46,41\},\quad y_{53}&=&22 \quad W_{54}&=&\{47,42\},\quad y_{54}&=&24\quad \quad\\  
W_{63}&=&\{48,43\},\quad y_{63}&=&27 \quad W_{64}&=&\{49,44\},\quad y_{64}&=&29 \quad \quad
\end{align*}
The $(mn, n, \sC_{\ell})$-DF we obtain from this choice consists of the following eight cycles.
\begin{align*}
C_{51} &=(0,58,7,43,12,-1,-2 )&\quad & C_{52}=(0,57,5,42,10,-1,-4)\\
C_{61} &=(0,69,8,46,13,-1,-17 )&\quad & C_{62}=(0,68,6,45,11,-1,-19)\\
C_{53} &=(0,56,3,49,8,1,-21 )&\quad & C_{54}=(0,59,5,52,10,1,-23)\\
C_{63} &=(0,67,4,52,9,1,-26 )&\quad & C_{64}=(0,66,2,51,7,1,-28).
\end{align*}
\end{ex}

\begin{lemma}\label{lemma:dividesn:2} 
Let $\ell = 2\lambda+3\geq 5$ be odd, let $m\ge4$ be even and 
$n=4\ell\nu$ with $\nu\geq1$. 
Then there exists a partition of $[1, mn/2]\setminus ([1,n/2]\cdot m)$ into ten subsets,
$A_i, A_i^*$ for $i\in\{-2, -1,1\}$, and $B, Y_1, Y_2, W$, satisfying the following properties:
\begin{enumerate}
  \item \label{lemma:dividesn:2:1}
  $|A_{-2}| = 2\nu$, $|A_{-1}| = (m-4)\nu$, $|A_{1}| = m\nu$, $|B|= 2(m-1)\nu$,
  $|Y_1|=4\nu(m-2)$, $|Y_2|=4\nu$,  $|W|= 2(\ell-5)(m-1)\nu$;
  \item \label{lemma:dividesn:2:2}
  there is a bijection $a\in A_i\mapsto a^* \in A_i^*$ for $i\in\{-2, -1,1\}$ such that 
  $B-(\lambda-1)m= \big\{a^*-a+i\mid i\in\{-2, -1,1\}, a\in A_i\big\}$;
  \item \label{lemma:dividesn:2:2b} $a^*-a\geq 2$ for every $a\in A_{-2}\cup A_{-1}\cup A_{1}$;
  \item \label{lemma:dividesn:2:3}
  $Y_j$ can be partitioned into pairs at distance $j$, for $j\in\{1,2\}$;
  \item \label{lemma:dividesn:2:4}
  $W$ can  be partitioned into $2(m-1)\nu$ sets  $\{W_a\mid a\in A_{-2}\cup A_{-1}\cup A_{1}\}$ each of size $\ell-5$ such that  
  \begin{enumerate}
      \item  $W_a = \big\{w_{a, t}, w_{a, t}-m \mid t\in[1, \lambda-1]\big\}$, and
      \item  $a> w_{a, t} \geq w_{a, t+1} + 2m$ for every $t\in[1, \lambda-2]$.
  \end{enumerate}
\end{enumerate}
\end{lemma}
\begin{proof}   Let $\epsilon=0$ or $1$ according to whether $\lambda$ is even or odd, and set 
 $q=2(m-1)\nu$ and $\mu=(2-\epsilon)m$.  We start by defining the intervals $I_h, J_h$ and the integer $\tau_{h}$, for 
 $h\in\{-2,-1,1\}\times\{1,2\}$, as follows:
\[
  J_h = \big[ 2^\epsilon(\ell-2)\nu, 2^\epsilon(\ell-2)\nu+\nu-1\big],\;\;\text{and}
\]
\[
   \begin{array}{|c|c|c|c|c|}
     \hline \rule{0pt}{1\normalbaselineskip}
        h     & \multicolumn{2}{|c|}{I_h}                            
              & \multicolumn{2}{|c|}{\tau_h}  \\[0.5ex] 
     \hline \hline \rule{0pt}{1\normalbaselineskip}
       (-1,1) &  \multicolumn{2}{|c|}{\left[\frac{m}{2}+2, m-1\right]} 
              &  \multirow{8}{*}{$\mu\nu\; + $}  
              & -\frac{m}{2} -1  \\[0.5ex]  
     \cline{1-3}  \cline{5-5} \rule{0pt}{1\normalbaselineskip}
        (1,1) & \multicolumn{2}{|c|}{\left[1, \frac{m}{2}+1\right]}    
              &&  +\frac{m}{2} -2  \\ [0.5ex] 
     \cline{1-3}  \cline{5-5} \rule{0pt}{1\normalbaselineskip}
       (-1,2) & \multirow{5}{*}{$(2\nu)^\epsilon m\; +$} 
              &\left[\frac{m}{2}+1,m-2\right] 
              &&  -\frac{m}{2}  \\[0.5ex] 
     \cline{1-1} \cline{3-3} \cline{5-5}\rule{0pt}{1\normalbaselineskip}
       (1,2) & 
             &\left[1,\frac{m}{2}-1\right]  
             &&   +\frac{m}{2}-1 \\[0.5ex]  
     \cline{1-1} \cline{3-3} \cline{5-5} \rule{0pt}{1\normalbaselineskip} 
      (-2,1) &   
             & \{\frac{m}{2}\}          
             &&  -1   \\[0.5ex]  
     \cline{1-1} \cline{3-3} \cline{5-5} \rule{0pt}{1\normalbaselineskip}
      (-2,2) &   
             & \{m-1\}                      
             && 0    \\[0.5ex]  
     \hline
   \end{array}
\]
 For $h\in\{-2,-1,1\}\times\{1,2\}$, set $A_{h} =  I_{h} + J_{h}\cdot \mu$
 and $A^*_{h} = A_h + \tau_h$. 
 Also, let $B$ and $Y_j=Y'_j \cup Y''_j$, for $j\in\{1,2\}$, the sets defined as follows: 
 \begin{align*}
   B   &= [1,2\nu]_o\cdot\mu + (\lambda-1)m + \big([-m+1,m-1]\setminus\{0\}\big),\\
   Y'_j&= [1,2\nu]_o\cdot\mu + (\lambda-1)m +2m{\nu}^\epsilon +
   \begin{cases}
     [-m+1, -2]\cup [2,m-1] & \text{if $j=1$},\\
     \{-1, 1\} & \text{if $j=2$},\\     
   \end{cases}\\
   Y''_j&= [1,2\nu]_o\cdot{m} + (4\nu +\lambda -\epsilon)m +
   \begin{cases}
     [-m+1, -2]\cup [2,m-1] & \text{if $j=1$},\\
     \{-1, 1\} & \text{if $j=2$}.    
   \end{cases}
 \end{align*}
 It is tedious but not difficult to check that 
 \begin{equation}\label{A'_h}
   \text{the sets $A_h$, $A^*_{k}$, $B$, $Y_1$ and $Y_2$ are pairwise disjoint.}
 \end{equation} 
 Also, by Lemma \ref{lem:AA*}, there is a bijection $a\in A_h \mapsto a^* \in A^*_h$  such that
 \begin{align*}
  \{a^* - a\mid a\in A_h\} &= [1, 2\nu]_o\cdot \mu +
     \begin{cases}
       \left[-m +2, -4 \right]_e 
       &    \text{if $h=(-1,1)$}, \\  
       \left[-2, m-2\right]_e 
       &  \text{if  $h=(1,1)$}, \\  
       \left[-m+3, -3\right]_o 
       &    \text{if $h=(-1,2)$}, \\ 
       \left[1, m-3\right]_o
       &  \text{if  $h=(1,2)$}, \\ 
       \{-1\}
       &  \text{if  $h=(-2,1)$}, \\  
       \{0\}      
       & \text{if  $h=(-2,2)$}.        
\end{cases}
\end{align*}
Recalling \eqref{A'_h}, it is not difficult to check that the sets $B$, $Y_1$, $Y_2$,
$A_{i} = A_{(i,1)}\ \cup\ A_{(i,2)}$ and $A^*_{i} = A^*_{(i,1)}\ \cup\ A^*_{(i,2)}$, for $i\in\{-2, -1, 1\}$, 
satisfy conditions \ref{lemma:dividesn:2:1}--\ref{lemma:dividesn:2:2b}.
Furthermore, since $Y_1$ has size $4\nu(m-2)= 2|A_{-1} \cup A_{1}|$ and 
it is the disjoint union of $4\nu$ intervals of size $m-2\equiv 0\pmod{2}$, $Y_1$ can be partitioned into 
$|A_{-1} \cup A_{1}|$ pairs of consecutive integers, hence 
$Y_1=\{y_a, y_a-1 \mid  a\in A_{-1}\cup A_1\}$. Similarly, since $Y_2$ is the disjoint union of 
$2\nu = |A_{-2}|$ pairs at distance two, we can write $Y_2=\{y_a, y_a-2 \mid  a\in A_{-2}\}$;
hence, condition \ref{lemma:dividesn:2:3} holds.

Finally, denoting by $W$ the set of all elements of $[1, mn/2]\setminus ([1,n/2]\cdot m)$ not lying in any of the sets
defined above, we have that $W$ has size $2(\ell-5)(m-1)\nu = (\ell-5)|A_{-2}\cup A_{-1}\cup A_{1}|$. Also, 
$W = (U_1\ \cup\ U_2)m + [1,m-1]$,
where $U_1 = [0, \lambda -\epsilon -1]$ and $U_2 = [\lambda - \epsilon +6\nu, (2\ell-4)\nu-1]$.
Since $U_1$ and $U_2$ have even size, and $|U_1\ \cup\ U_2| = 2(\ell-5)\nu$, there exists a partition
$\{W_a\mid a\in A\}$ of $W$ satisfying condition \ref{lemma:dividesn:2:4}, and this completes the proof.
\end{proof}

\begin{prop} \label{part size 0 mod 4ell - m even}
Let $\ell \geq 5$ be odd, and let $n\equiv 0 \pmod{4\ell}$. 
There exists a $(mn,n,{\sC}_\ell)$-DF for every even $m\geq 4$.
\end{prop}
\begin{proof}
Set $D=[1, mn/2]\setminus ([1,n/2]\cdot m)$, let $n=4\ell\nu$ with $\nu\geq1$, and set 
$q=2(m-1)\nu$. Also, set $\lambda=(\ell-3)/2$ and
let $\epsilon=0$ or $1$ according to whether $\lambda$ is even or odd.

By Lemma \ref{lemma:dividesn:2}, there is a partition of
$D=[1, 4m\nu]\setminus ([1,4\nu]\cdot m)$ into ten subsets
$A_i, A_i^*$ for $i\in\{-2, -1,1\}$, and $B, Y_1, Y_2, W$ satisfying the 
conditions \ref{lemma:dividesn:2:1}--\ref{lemma:dividesn:2:4} of Lemma \ref{lemma:dividesn:2}.

Set $A= A_{-2}\ \cup\ A_{-1}\ \cup\ A_{1}$
and let $\mathcal{C}=\{C_a\mid a\in A\}$, with 
$C_a = (c_{a,0}, c_{a,1}, \ldots, c_{a,\ell-1})$, 
be a set of 
$q$ closed trails of length $\ell$ defined as follows:
\begin{align*}
(c_{a,0}, c_{a,1}, c_{a,2}) &= (0,a^*, a^*-a),\\
(c_{a,\ell-2}, c_{a,\ell-1}) &= 
\begin{cases}
 (1, -y_a + 1) & \text{if $a\in A_{-1}$},\\
 (-1, -y_a)    & \text{if $a\in A_{1}$}, \\
 (2, -y_a + 2) & \text{if $a\in A_{-2}$}, \\
\end{cases}\\
c_{a, u} &= a^*-a +
\begin{cases}
  w_{a, \frac{u-1}{2}} + \frac{u-3}{2}m     & \text{if $u\in[3, \ell-4]$ is odd}, \\
  \frac{u-2}{2}m       & \text{if $u\in[4, \ell-3]$  is even},  
\end{cases}
\end{align*}
We claim that $\mathcal{C}$ is the desired set of base cycles, that is, 
$\Delta \mathcal{C} = \mathbb{Z}_{mn}\setminus(m\mathbb{Z}_{mn})$ and the vertices of each $C_a$ are pairwise distinct. For every $i\in\{-2, -1,1\}$ and $a\in A_{i}$, we have that
\begin{align*}
  \Delta C_a &= \pm\left\{a, a^*, y_a, y_a-|i|, c_{a,\ell-3} -  c_{a,\ell-2}\right\}\ \cup\ W_a \\
             &= \pm\left\{a, a^*, y_a, y_a-|i|, a^*- a + (\lambda-1) m + i\right\}\ \cup\ W_a.
\end{align*}  
By conditions \ref{lemma:dividesn:2:1}-\ref{lemma:dividesn:2:4}, it follows that 
$\Delta \mathcal{C} = \pm D = \mathbb{Z}_{mn}\setminus(m\mathbb{Z}_{mn})$.

Finally considering that,
by conditions \ref{lemma:dividesn:2:2b} and \ref{lemma:dividesn:2:4} 
of Lemma \ref{lemma:dividesn:2},
$m < w_{a,1} < a$, $w_{a,t}\geq w_{a,t+1} + 2m$,  and $a^* - a \geq 2$, 
for every $a\in A$ and $t\in [1, \lambda-2]$,
it follows that
\begin{align*}
  mn/2> a^* =c_{a,1} &> a^*-a + w_{a,1} = c_{a,3} >  c_{a,5} > \ldots > c_{a,\ell-6} \\
    &> c_{a,\ell-4} =  a^*-a + w_{a,\lambda-1} + (\lambda-2)m \\
  &> a^*-a  +  (\lambda-1)m = c_{a, \ell-3} > c_{a, \ell-5} > \ldots > c_{a,6}  \\
  &> c_{a,4} = a^*-a + 2m > a^*-a = c_{a,2}\geq 2
\end{align*}
and this guarantees that each $C_a$ is a cycle.
\end{proof}

\begin{ex} Take $\ell=7, m=4$ and $n=28$, so that so that $\nu=1, \lambda=2, q=6$.
We have 
\begin{align*}
A_{(-1,1)} & = &\varnothing , \quad  A_{(-1,1)}^* & =&\varnothing , \quad A_{(1,1)} & = &\{41,42,43\}, \quad A_{(1,1)}^* & = &\{49,50,51\}, \\
A_{(-1,2)} & = &\varnothing, \quad  A_{(-1,2)}^* & =&\varnothing, \quad A_{(1,2)} & = &\{45\}, \quad A_{(1,2)}^* & = &\{54\}, \\
A_{(-2,1)} & = &\{46\}, \quad A_{(-2,1)}^* & = &\{53\}, \quad A_{(-2,2)} & = &\{47\}, \quad A_{(-2,2)}^* & = &\{55\},
\end{align*}
so that $A_1=\{41,42,43,45\}$, $A_1^*=\{49,50,51,54\}$, $A_{-2}=\{46,47\}$, $A_{-2}^*=\{53,55\}$, while in this case $A_{-1}=A_{-1}^*=\varnothing$.

Also, $B=[9,11]\cup[13,15]$, $Y_1'=[17,18]\cup[22,23]$, $Y_1''=[25,26]\cup[30,31]$, and $Y_2'=\{19,21\}$, $Y_2''=\{27,29\}$,
so that $W=([1,7]\setminus \{4\})\cup ([33,39]\setminus \{36\}) $. For the sets $W_a$ 
and elements $y_a$, $a\in A_1\cup A_{-2} (\cup A_{-1})$ we can choose for instance\begin{align*}
W_{41}&=&\{5,1\},\quad y_{41}&=&18  \quad W_{42}&=&\{6,2\},\quad y_{42}&=&23   \quad \quad\\
W_{43}&=&\{7,3\},\quad y_{43}&=&26 \quad W_{45}&=&\{37,33\},\quad y_{45}&=&31 \quad \quad\\  
W_{46}&=&\{38,34\},\quad y_{46}&=&19 \quad W_{47}&=&\{39,35\},\quad y_{47}&=&27 \quad \quad 
\end{align*}
The $(mn, n, \sC_\ell)$-DF we obtain from this choice
consists of the following six cycles.
\begin{align*}
C_{41} &=(0,51,10,15,14,-1,-18 )&\quad & C_{42}=(0,50,8,14,12,-1,-23)\\
C_{43} &=(0,49,6,13,10,-1,-26 )&\quad & C_{45}=(0,54,9,46,13,-1,-31)\\
C_{46} &=(0,53,7,45,11,2,-19 )&\quad & C_{47}=(0,55,8,47,12,2,-27)
\end{align*}
\end{ex}

\section{Concluding remarks}\label{conclusion}
Recall that Corollary~\ref{nonexistence corollary} gives certain definite exceptions to the existence of a cyclic $\ell$-cycle system of $K_m[n]$.  The reader may wonder if these exceptions can be ruled out if we consider regular
$\ell$-cycle systems  under a group $G$ which is not necessarily cyclic. 
The following two results will partially answer this question. 
The first provides us with a necessary condition for the existence of 
a $G$-regular $\ell$-cycle system of $K_m[n]$ under the assumption that the $G$-stabilizer of
each cycle has odd order.

\begin{thm}\label{nonexistence G-regular} Let $\cB$ be a $G$-regular $\ell$-cycle system of $K_m[n]$. If each cycle of $\cB$ has a $G$-stabilizer of 
odd order, then either $m\not\equiv 2,3 \pmod{4}$ or $n\not\equiv 2 \pmod{4}$. 
\end{thm}
\begin{proof} We assume for a contradiction that $m\equiv 2,3 \pmod{4}$ and $n\equiv 2 \pmod{4}$, hence $|G|\not\equiv 0 \pmod{8}$.
Let $\cB$  be a $G$-regular $\ell$-cycle system of $K_m[n]$. Without loss of generality, we can assume  that
  \begin{enumerate}
    \item $K_m[n] = \Cay[G:G\setminus N]$ where $N$ is a subgroup of $G$ of order $n$, and 
    \item $C+g\in \cB$ for every $C\in\cB$ and $g\in G$. 
  \end{enumerate}
  
  We first show that every element of $G$ of order $2$ (i.e., involution of $G$) belongs to $N$. In fact, if $G\setminus N$ contains an element $y$ of order $2$, then the edge $\{0,y\}$ must be contained in some cycle of $\cB$, say $C$. Hence the edge $\{0,y\}$ belongs to $C+y$, which is still a cycle of $\cB$. Since $\cB$ is a cycle system of $K_m[n]$, every edge of $K_m[n]$ is contained in exactly one cycle of  $\cB$.  Therefore $C+y = C$, meaning that $y$ belongs to the $G$-stabilizer of $C$, which therefore has even order in contradiction to our assumption. 
  
  We now show that a Sylow $2$-subgroup of $G$, say $P$, is cyclic. If $|G|\equiv 2 \pmod{4}$, then $|P|=2$, and hence $P$ is cyclic. Since $|G|\not\equiv 0 \pmod{8}$, it is left to consider the case where  $|G|\equiv 4 \pmod{8}$, hence $P$ is either cyclic or isomorphic to $\Z_2\times \Z_2$. But in the latter case, all non-zero elements of $P$ have order $2$, hence $P$ is a subgroup of $N$ which therefore has order divisible by $4$ contradicting the assumption. We have thus proven that all Sylow $2$-subgroups of $G$ are cyclic.
  
  From the above arguments, we can prove that $G$ has a subgroup of index $2$. Indeed,
  since all Sylow $2$-subgroups $P$ are cyclic we can apply the Cayley normal 2-complement theorem, so that $G$ has a normal subgroup $S$ of order $|G|/|P|$ with $G=P+S$.
  Since the factor group $G/S$ is isomorphic to $P$, it is cyclic so  it has a subgroup
  $H/S$ of index $2$, and $H$ is therefore a subgroup of $G$ of index $2$. 
  
  Also, if $H$ has even size, then it contains all the involutions of $G$. Indeed, denoting by $y$ any element of $G$ of order $2$, then $y$ belongs to a suitable Sylow $2$-subgroup of $G$, say $Q$. Since $Q$ is cyclic, $y$ is the only involution of $Q$. Considering that $H\cap Q$ is a Sylow $2$-subgroup of $H$, then $|H\cap Q|\geq 2$ is even, hence $y\in H\cap Q$. 
  
  Finally, recalling that $|G|=mn$ with $m\equiv 2,3 \pmod{4}$ and $n\equiv 2 \pmod{4}$, we can show that
  \begin{equation}\label{N/(H cap N)}
    |N/(H\cap N)| = 
    \begin{cases} 
      1 & \text{if $m\equiv 2 \pmod{4}$}, \\
      2 & \text{if $m\equiv 3 \pmod{4}$}.
    \end{cases}
  \end{equation} 
  Since $H$ has index $2$ in $G$, then
$|N/(H\cap N)|\in\{1,2\}$. 
  If $m\equiv 2\pmod{4}$, then $|H|\equiv 2\pmod{4}$. Since $|N|=n\equiv 2\pmod{4}$, 
  by the Cayley normal 2-complement theorem we have that $N= N' + \{0, y\}$ where $N'$ is a subgroup of index $2$ and $y$
  is any involution of $N$. Since $H$ contains all  the involutions of $G$, and $|N'|=|N|/2$ is odd , then
  $N', \{0, y\}\subset H$, that is, $N= N' + \{0, y\} \subset H$; therefore $H\cap N = N$ and 
  $|N/(H\cap N)|=1$.
  If $m\equiv 3\pmod{4}$, then 
  $|H|$ is odd and $|N/(H\cap N)|=2$.
  
  Let $\cF=\{C_1, C_2, \ldots, C_t\}$ be a complete system of representatives for the $G$-orbits of $\cB$, 
  let $S_i=\{g\in G\mid C_i+g=C_i\}$ be the $G$-stabilizer of $C_i$, 
  and set $s_i=|S_i|$ for $i\in[1,t]$. 
  Since by assumption $s_i$ is odd, and recalling that
  the automorphism group of an $\ell$-cycle is the dihedral group $\mathbb{D}_{2\ell}$ of size $2\ell$, 
  then each $S_i$ is isomorphic to a
  subgroup of $\mathbb{D}_{2\ell}$ and $s_i$ is a divisor of $\ell$, for $i\in[1,t]$.  
  Also, considering that all subgroups of 
  $\mathbb{D}_{2\ell}$ of odd size are cyclic, then each $S_i$ is cyclic. Therefore, letting $\lambda_i = \ell/s_i$ and
  $C_i=(c_{i,0}, c_{i,1}, \ldots, c_{i,\ell_i-1})$, we have that 
  \begin{equation}\label{C+x_i=C}
  c_{i, a\lambda_i + b} = c_{i, b} + ax_i
  \end{equation}
  for every $a\in [0, s_i-1]$ and $b\in [0, \lambda_i-1]$,
  where $x_i$ is a suitable generator of $S_i$.
  
  Now set $D_i= \{\delta_{i,j}\mid j\in[0,\lambda_i-1]\}$ 
  where $\delta_{i,j} = c_{i, j+1} - c_{i,j}$ 
  for every $i\in[1,t]$ and $j\in [0, \lambda-1]$. Since every edge 
  of $K_m[n]= \Cay[G:G\setminus N]$ is contained in exactly one cycle of 
  $\cB$ and recalling that any translation preserves the differences,
  it follows that 
  \begin{equation}\label{partialdifferences}
   \text{$\big\{D_i, - D_i \mid i\in[1,t]\big\}$ is a partition of $G\setminus N$}.
  \end{equation} 
  Also, by \eqref{C+x_i=C} it follows that
  $\delta_{i,\lambda_i} + \delta_{i,\lambda_i-1} + \ldots, \delta_{i,0} + c_{i,0} = 
  c_{i, \lambda_i} = c_{i,0} + x_i$.
  Since $x_i$ has odd order, it follows that $x_i \in H$. Considering that $G/H$ is abelian 
  (since it has order $2$), the following equality involving cosets of $N$ holds:  
  \begin{equation}
  \sum_{j=0}^{\lambda_i} \delta_{i,j} + H = x_i + H = H.
  \end{equation}
  This means that $\sum_{j=0}^{\lambda_i} \delta_{i,j}\in H$. In other words, 
  each $D_i$ contains an even number of elements belonging to 
  $G\setminus H$; hence, by \eqref{partialdifferences} it follows that 
  $|~(~G~\setminus~N)~\setminus~H| \equiv 0 \pmod{4}$.
  However, by \eqref{N/(H cap N)} it follows that $|(G\setminus N)\setminus H| = \lfloor m/2\rfloor n \equiv 2 \pmod{4}$ which is a contradiction.
\end{proof}

It follows that a regular $\ell$-cycle system of $\Kmn$ over a non-cyclic group $G$ and satisfying condition 2 of Corollary \ref{nonexistence corollary} must necessarily contain 
cycles with non-trivial $G$-stabilizers of even size. On the contrary, regular 
$\ell$-cycle systems of $\Kmn$ satisfying condition 1 of Corollary \ref{nonexistence corollary} do not exist, as shown below.

\begin{cor} Let $G$ be an arbitrary group of order $mn$. Then there is no $G$-regular $\ell$-cycle system of $K_m[n]$ whenever  $\ell$ is odd, $m\equiv 2,3 \pmod{4}$, and $n\equiv 2 \pmod{4}$.
\end{cor}
\begin{proof} Since $\ell$ is odd, it is easy to note that the $G$-stabilizer of any cycle of a $G$-regular $\ell$-cycle system $\cB$ has odd size. Indeed, an involution of $G$ fixing an $\ell$-cycle of 
$\cB$ must fix one of its vertices contradicting the assumption that $G$ acts sharply transitively on the vertex set. Then the assertion follows from Theorem \ref{nonexistence G-regular}.
\end{proof}

\section*{Acknowledgements}

A.C.\ Burgess gratefully acknowledges support from an NSERC Discovery Grant.
F. Merola and T. Traetta gratefully acknowledge support from GNSAGA of Istituto Nazionale di Alta Matematica.

\end{document}